\documentclass[12pt,a4paper]{amsart}
\usepackage{fullpage} 
\usepackage{amsmath}
\usepackage{amsthm}
\usepackage{amssymb}
\usepackage{mathtools}
\usepackage[only,mapsfrom]{stmaryrd}
\usepackage{graphicx}
\usepackage{cancel} 
\usepackage{color} 
\usepackage[all]{xy} 
\usepackage{tikz}
\usetikzlibrary{arrows,calc}
\usetikzlibrary{cd}
\usepackage{nicefrac}
\usepackage{paralist}
\usepackage{mathrsfs}
\usepackage{enumitem}
\usepackage[colorlinks,citecolor=blue,linkcolor=blue,urlcolor=blue,filecolor=blue,breaklinks]{hyperref}

\numberwithin{equation}{section}
\numberwithin{figure}{section}

{\begin{compactitem}

}%
{\end{compactitem}}

\newtheorem{thm}{Theorem}[section]
\newtheorem{athm}{Theorem}

\newtheorem{lem}[thm]{Lemma}
\newtheorem{prop}[thm]{Proposition}
\newtheorem{cor}[thm]{Corollary}

\newtheorem*{conj*}{Conjecture}

\theoremstyle{definition}
\newtheorem{rem}[thm]{Remark}
\newtheorem{defn}[thm]{Definition}
\newtheorem{ex}[thm]{Example}

 \newcommand{\BN}{{\mathbb {N}}}
 
\newcommand{\BQ}{{\mathbb {Q}}}

 \newcommand{\BZ}{{\mathbb {Z}}}

\newcommand{\BSol}{{\mathcal{BS}}}

\newcommand{\incl}[3][right]%
{%
\draw[<-,>=#1 hook] #2 to ($ #2!0.5!#3 $);
\draw[->] ($ #2!0.5!#3 $) to #3;%
}

\newcounter{commentcounter}

\title{On the Boone--Higman Conjecture for groups acting on locally finite trees}

\author{Kai-Uwe Bux}
\address{Faculty of Mathematics, Bielefeld Univeristy, D-33501 Bielefeld, Germany}
\email{bux@math.uni-bielefeld.de}

\author{Claudio Llosa Isenrich}
\address{Faculty of Mathematics, Karlsruhe Institute of Technology, Englerstra\ss e 2, 76131 Karlsruhe, Germany}
\email{claudio.llosa@kit.edu}

\author{Xiaolei Wu}
\address{Shanghai Center for Mathematical Sciences, Jiangwan Campus, Fudan University, No.2005 Songhu Road, Shanghai, 200438, P.R. China}
\email{xiaoleiwu@fudan.edu.cn}

\date{July 2024}

\begin{document}

\begin{abstract}
    We develop a method for proving the Boone--Higman Conjecture for groups acting on locally finite trees. As a consequence, we prove the Boone--Higman Conjecture for all Baumslag--Solitar groups and for all free(finite rank)-by-cyclic groups, solving it in two cases that have been raised explicitly by Belk, Bleak, Matucci and Zaremsky. We also illustrate that our method has applications beyond these cases and may offer a route for proving the Boone--Higman Conjecture for many classes of groups.

\end{abstract}

\maketitle

\section*{Introduction}

A classical result in group theory is the Higman Embedding Theorem \cite{Hig-61}, which says that a group is recursively presentable if and only if it embeds in a finitely presentable group. This naturally raises the question if other algorithmic properties of groups admit similar characterisations. One fundamental property of this kind is the solvability of the word problem for finitely generated groups. The Boone--Higman Conjecture provides the following conjectural characterisation of this property.
\begin{conj*}[{Boone--Higman}]\label{conj:Boone--Higman}
    A finitely generated group has solvable word problem if and only if it embeds in a finitely presented simple group.
\end{conj*}

It was first observed by Kuznetsov  that a finitely generated subgroup of a finitely presented simple group always has solvable word problem \cite{Kuznetsov58}. The other direction however is far from understood. What Boone and Higman were able to show is that a finitely generated group has solvable word problem if and only if it embeds into a simple subgroup of a finitely presented group \cite{BoHi74}. On the other hand, the conjecture has now been verified for many classes of groups, including virtually special groups, hyperbolic groups \cite{BBMZ-hyperbolic-23} and finitely generated subgroups of $GL_n(\BQ)$ \cite{Zar-24a},  see \cite[Section 5]{BBMZ-survey-23} for a summary of the current status of the conjecture. In this work we develop techniques to verify the Boone--Higman Conjecture for large classes of groups that act on locally finite trees. As a consequence we obtain:
\begin{athm}\label{mainthm:Boone-Higman}
    The Boone--Higman Conjecture holds for the following classes of groups:
    \begin{enumerate}
        \item Baumslag--Solitar groups;
        \item free(finite rank)-by-cyclic groups;
        \item Leary--Minasyan groups.
        \item Euclidean triangle Artin groups.
    \end{enumerate}
\end{athm}
We emphasize that the problem of proving the Boone--Higman Conjecture for Baumslag--Solitar groups and free-by-cyclic groups was raised explicitly by Belk, Bleak, Matucci and Zaremsky \cite[Problem 5.3(4) \& (8)]{BBMZ-survey-23}.

More generally,  given a group acting on a tree, our techniques allow us to establish the following  theorem.

\begin{athm}\label{thm:main-actions-on-trees} [\ref{thm:main-graphs-of-groups}]
    Let  $G$ be a group acting faithfully and cocompactly on a locally finite tree. If all edge stabilizers for the action are finitely presented, then $G$ embeds in a finitely presented simple group.
\end{athm}

Theorem~\ref{thm:main-actions-on-trees} is proved by first embedding $G$ into a bigger group which acts oligomorphically (see Definition \ref{defn:oligmc}) on the vertex set of the tree and then embedding this group further into the associated twisted Brin--Thompson group \cite{BelZar-20} which is always simple. In fact, given any group acting faithfully  on a locally finite tree $\mathcal{T}$, we define a new group $\mathrm{RP}_G(\mathcal{T})$, the rigid permutation group for the action of $G$ on $\mathcal{T}$, which is one of our key innovations and might be of independent interest. The definition is partially inspired by that of a topological full group \cite{GTS99,Mat12}. But instead of acting on the boundary of the tree $T$, our group is a subgroup of the permutation group of its vertex set. More precisely,  $\mathrm{RP}_G(\mathcal{T})$ is the subgroup of  permutations of the vertex set of $\mathcal{T}$ such that outside a finite subtree of $\mathcal{T}$, the action on the vertices of each component coincides with the action of some element of $G$. In particular, $\mathrm{RP}_G(\mathcal{T})$ contains both $G$ and the group of compactly supported permutations of vertices of $\mathcal{T}$. There is an induced action of  $\mathrm{RP}_G(\mathcal{T})$ on the boundary of $\mathcal{T}$ which coincides with the topological full group for the action of $G$ on the boundary. 

The hard part of our work lies in establishing the finiteness properties of $\mathrm{RP}_G(\mathcal{T})$ in nice situations. In fact, modulo some technical assumptions we show that $\mathrm{RP}_G(\mathcal{T})$  is of type $F_n$ if the edge stabilizers are of type $F_n$, see Proposition \ref{prop:main-strongly-faithful-actions-on-trees}. To prove this result, we build a Stein--Farley CAT(0) cube complex  equipped with a height function for $\mathrm{RP}_G(\mathcal{T})$ and show that the descending links of large height are highly connected. While this approach is used to prove finiteness properties for many Thompson-like groups, our proof is intricate and required several new ideas, for both, producing a suitable Stein--Farley complex and then analysing the connectivity properties of its descending links. For instance, the latter involved a surprising use of Dickson's Lemma \cite{FigFigSchSch-11} about unions of orthants in $\mathbb{N}_0^k$ (see the proof of Lemma~\ref{lem:maximum-in-alternative-histories}).

Note that the groups in Theorem~\ref{mainthm:Boone-Higman} all have a natural graph of groups structure, but they do not necessarily act faithfully on the corresponding Bass--Serre trees. Our idea is to enlarge the graphs of groups, so that the Bass-Serre trees remain locally finite but the actions become faithful. This idea leads us to define a very general class of graphs of groups for any given group $G$, which we call generalised Baumslag--Solitar groups  $\mathcal{BS}_G$ over $G$. By definition, $\BSol_G$ is the class of groups consisting of all finite graphs of groups with the property that all edge and vertex groups are abstractly commensurable with $G$ and all edge group inclusions in vertex groups have finite index. We have the following theorem for $\mathcal{BS}_G$.

    \begin{athm}\label{mthmC}[\ref{thm:main-BS}]
        Let $G$ be a finitely presented group. Assume that there is a non-trivial group $H\in \BSol_{G}$ such that $H$ acts faithfully on its Bass--Serre tree. Then every $K\in \BSol_{G}$ embeds in a finitely presented simple group.  
    \end{athm}

Theorem~\ref{mthmC} can be applied to many classes of groups. For example, Theorem \ref{mainthm:Boone-Higman} is proved by taking $G$ to be a finitely generated abelian group or a finitely generated free group. We will discuss this in detail in Section \ref{sec:applications}. 
We also emphasize that the only obstruction for applying Theorem~\ref{mthmC} that we are currently aware of is that any $G$ satisfying its hypotheses must be residually finite. However, even if $G$ is residually finite, groups in $\BSol_G$ need not themselves be, as is shown by $BS(2,3)\in \BSol_{\mathbb{Z}}$ and the Burger--Mozes groups, which are in $\BSol_{\mathbb{F}_2}$. We thus expect that Theorem~\ref{mthmC} can be applied to prove the Boone--Higman Conjecture for further interesting classes of groups. A further interesting consequence of our approach is that the (finitely presented simple) Burger--Mozes groups also satisfy the permutational Boone--Higman Conjecture; see Remark \ref{rem:burger-mozes} for details.

\subsection*{Acknowledgements.}

KUB thanks the Deutsche      Forschungsgemeinschaft (DFG, German Research Foundation) for support     via the CRC TRR 358 – Project-ID 491392403. CLI thanks the Deutsche Forschungsgemeinschaft (DFG, German Research Foundation) for support via the RTG 2229 - Project-ID 281869850. XW is currently a member of LMNS and supported by NSFC No.12326601. We thank Bielefeld University and the Shanghai center for Mathematical Sciences, where part of the work was carried out, for the hospitality.

\section{Rigid permutation groups}

In this section we define the rigid permutation group $\mathrm{RP}_G(\mathcal{G})$ for a group $G$ acting on a locally finite graph $\mathcal{G}$. We start by introducing some graph theoretical notation that we will use throughout the paper. Given a graph $\mathcal{G}$ and a subgraph $\mathcal{H}\subseteq \mathcal{G}$, we will denote by $\mathcal{G}\setminus \mathcal{H}$ the subgraph of $\mathcal{G}$ spanned by all vertices in $V(\mathcal{G})\setminus V(\mathcal{H})$ and by $\overline{\mathcal{G}\setminus \mathcal{H}}$ the minimal subgraph of $\mathcal{G}$ containing all edges in $E(\mathcal{G})\setminus E(\mathcal{H})$. More generally, we will sometimes also denote by $\overline{M}$ the minimal subgraph of $\mathcal{G}$ containing all edges in a subset $M\subseteq E(\mathcal{G})$. We will further denote by ${\rm int}(\mathcal{G})$ the subgraph of $\mathcal{G}$ spanned by all interior vertices of $\mathcal{G}$, where an interior vertex is a vertex that is not a leaf. 

\begin{defn}
Let $\mathcal{G}$ be a locally finite infinite graph. We call a subgraph   a \emph{piece} of $\mathcal{G}$ if it is an infinite connected component of the complement of a finite subgraph of $\mathcal{G}$.
\end{defn}

\begin{defn}
 Let $G$ be a group and let $\mathcal{G}$ be a locally finite infinite graph which $G$ acts on. We call the action \emph{faithful} if no non-trivial element of $G$ acts as identity on $\mathcal{G}$. We call the action \emph{strongly faithful} if no non-trivial element of $G$ acts as identity on a piece of $\mathcal{G}$.
\end{defn}

\begin{defn}
    Let $G$ be a group and let $\mathcal{G}$ be a locally finite graph which $G$ acts on. We call a permutation $\phi$ of the vertices of $\mathcal{G}$ \emph{rigid} if  there is a finite subgraph $F$ of $\mathcal{G}$, such that for each piece $P$ of $\mathcal{G}\setminus F$, there exists an element $g\in G$ with $\phi\mid_P = g\mid_P$.  The \emph{rigid permutation group $\mathrm{RP}_G(\mathcal{G})$ (of the $G$-action on $\mathcal{G}$)} is the group of rigid permutations of $V(\mathcal{G})$. $F$ is called the \emph{(rigid) support} of $\phi$.

   Given a finite set $\mathcal{B}\subset V(\mathcal{G})$, we denote by $\mathrm{RP}_G(\mathcal{G}, \mathcal{B})$ the subgroup of $\mathrm{RP}_G(\mathcal{G})$ that fixes $\mathcal{B}$ pointwise.
\end{defn}

\begin{rem}
\begin{enumerate}
    \item When the graph $\mathcal{G}$ is finite, $\mathrm{RP}_G(\mathcal{G})$ coincides with the permutation group of vertices of $\mathcal{G}$.

    \item Any group element of $G$ acts on $\mathcal{G}$ as a rigid permutation. In particular, when the action is faithful, we have an embedding of $G$ into $\mathrm{RP}_G(\mathcal{G})$.

    \item 
    The action of $G$ on $\mathcal{G}$ induces an action of $G$ on $\partial \mathcal{G}$. One can then define the topological full group for the action of $G$ on $\partial \mathcal{G}$ once the induced action on $\partial \mathcal{G}$ is faithful. We expect that our proof for finiteness properties of $\mathrm{RP}_G(\mathcal{G})$ also applies to the corresponding topological full group.
    \item For a finite set $\mathcal{B}\subset V(\mathcal{G})$ one can view the group $\mathrm{RP}_G(\mathcal{G},\mathcal{B})$ as a rigid permutation group analogue of the pure mapping class group of a punctured surface. 
    \end{enumerate}
\end{rem}

\begin{ex}
    Let $G$ be a finitely generated infinite  group, and $C(G,S)$ be its Cayley graph with respect to some finite generating set $S$. Then $C(G,S)$ is a locally finite graph on which $G$ acts cocompactly and faithfully. It is not hard to see that the rigid permutation group $\mathrm{RP}_{G}(C(G,S))$ is again a finitely generated group. Indeed, denote the elements of $S$ by $s_1,\cdots s_n$, and denote by $\sigma_i$ the transposition of $1$ and $s_i$. Then $\mathrm{RP}_{G}(C(G,S))$ is generated by $s_1,\cdots, s_n$, $\sigma_1,\cdots, \sigma_n$. But $\mathrm{RP}_{G}(C(G,S))$ is in general not finitely presented. For example,  take $G = \BZ$ and $S=\{1\}$, then $\mathrm{RP}_{G}(C(G,S))$ is the second Houghton group which is not finitely presented \cite[Section 5]{Bro-87}. 
\end{ex}
  
\begin{rem}
  In the case of the Cayley graph $C(G,S)$, the rigid permutation group $\mathrm{RP}_{G}(C(G,S))$ can be also described as the group of quasi-automorphisms of the directed labeled Cayley graph. Quasi-automorphisms preserve colors and directions but are allowed to break finitely many edges. Quasi-automorphism groups and their finiteness properties have been studied by Lehnert in his PhD thesis~\cite{Lehnert2008}, see also \cite{AuAyFa18,NucStJG-18}.
\end{rem}

\section{A connectivity lemma}
\label{sec:connectivity-lemma}

We will later apply Bestvina--Brady Morse theory to derive that our groups have the asserted finiteness properties. For this we will need to check that the descending links in a cube complex that our groups act on have good connectivity properties. This argument will rely on a generalization of a result by Zaremsky \cite[Lemma 6.5]{Zar17} that we will now state and prove.

\begin{defn}
  Let $\Delta$ be a simplicial complex. A set $\sigma$ of $l+1$ vertices in $\Delta$ is called an \emph{$m$-pseudosimplex of dimension $l$} if any subset of up to $m+1$ vertices in $\sigma$ spans a simplex in $\Delta$. Note that any subset of an $m$-pseudosimplex is itself an $m$-pseudosimplex. We call the subsets of $\sigma$ its \emph{pseudofaces}.

  Two vertex sets $\sigma$ and $\tau$ in $\Delta$ are called \emph{$m$-joinable} if their union is an $m$-pseudosimplex, that is, if any subset of up to $m+1$ vertices of $\sigma\cup\tau$ spans a simplex in $\Delta$. Note that in this case, $\sigma$ and $\tau$ are both $m$-pseudosimplices.

  We say that $\Delta$ is \emph{$m$-flag with respect to an $m$-pseudosimplex $\sigma$} if any $m$-pseudosimplex $\rho$ of $\Delta$ and any pseudoface $\tau$ of $\sigma$ are $m$-joinable provided that they are vertex-wise joinable, that is, provided that any vertex of $\rho$ is joined to any vertex of $\tau$.
\end{defn}

\begin{lem}\label{lem:zaremsky-connectivity-lemma}
    Let $\Delta$ be a simplicial complex, and let $m, k\in \BN$. Suppose that there exists an $m$-pseudosimplex $\sigma$ of dimension $l$ such that $\Delta$ is $m$-flag with respect to $\sigma$, and for every vertex $v$ in $\Delta$, $v$ is $m$-joinable to some $(l-k)$-pseudoface of $\sigma$. Then $\Delta$ is $\min\{\lfloor \frac{l}{k}\rfloor-1,m-1\}$-connected.
\end{lem}

\begin{proof}
  
  The result is trivially true if $\sigma$ is the empty simplex, so assume $\sigma\neq \emptyset$. Note that the hypotheses carry over to any full subcomplex of $\Delta$ that contains $\sigma$. Since the image of a sphere inside $\Delta$ under a continuous map lives inside some full finite subcomplex containing $\sigma$, we may thus restrict to the case when $\Delta$ is finite. 
  
  We now induct on the number $n$ of vertices in $\Delta$ outside $\sigma$. More precisely, our induction hypothesis is that for a given $n$ the statement is true for all possible values of $m$, $k$ and $l$.

  If $\Delta=\sigma$, the complex is $(m-1)$-connected since it is an $m$-pseudosimplex. This proves the case when $n=0$. Now assume that the induction hypothesis holds for $n-1$.

  Consider a vertex $v\in\Delta$ outside $\sigma$. Let $\Delta'$ be the full subcomplex of $\Delta$ spanned by all vertices but $v$. Then $\sigma$ is contained in $\Delta'$ and the hypotheses of the lemma are satisfied. By induction, $\Delta'$ is $\min\{\lfloor \frac{l}{k}\rfloor-1,m-1\}$-connected. 

  Let $L$ be the relative link of $v$ in $\Delta'$. It suffices to show that $L$ is
 $\min\{\lfloor \frac{l}{k}\rfloor-2,m-2\}$-connected. This is again done by reducing the statement to our induction hypothesis. Care is required as $L$ is not the full subcomplex spanned by its vertices, which means that we require an argument to establish that the hypotheses of the lemma are satisfied for $L$ and $\sigma':=\sigma\cap L$.
 
 First, note that $v$ is vertex-wise joinable to $\sigma'$ and therefore (by the $m$-flag condition) $m$-joinable to the full subcomplex spanned by the vertices of $\sigma'$  in $\Delta$.  It follows that $\sigma'$ is an $(m-1)$-pseudosimplex in $L$. Also note that the dimension $l'$ of $\sigma'$ is at least $l-k$.
 
 We further observe that $L$ is $(m-1)$-flag with respect to $\sigma'$. Indeed, consider an $(m-1)$-pseudosimplex $\rho$ in $L$ and a pseudoface $\tau'$ of $\sigma'$ so that any vertex of $\rho$ and any vertex of $\tau'$ span an edge in $L$. Consider a total of up to $m$ vertices $u_1,\ldots,u_p$ of $\rho$ and $w_1,\ldots,w_q$ of $\sigma'$. Then $\{v,u_1,\ldots,u_p\}$ is a simplex in $\Delta$ that is vertex-wise joinable to $\{w_1,\ldots,w_q\}$. Hence $\{v,u_1,\ldots,u_p,w_1,\ldots,w_q\}$ spans a simplex in $\Delta$ whence $\{u_1,\ldots,u_p,w_1,\ldots,w_q\}$ spans a simplex in $L$.

 It remains to show that any vertex $w$ in $L$ outside $\sigma'$ is $(m-1)$-joinable to an $(l'-k)$-pseudoface of $\sigma'$. Note that, $w$ is joinable to all but at most $k$ vertices of $\sigma$ in $\Delta$. Hence, the edge from $w$ to $v$ is vertex-wise joinable in $\Delta$ to all but at most $k$ vertices of $\sigma'$. By the $m$-flag condition for $\sigma$ inside $\Delta$, the edge is $m$-joinable in $\Delta$ to a face $\tau'$ of $\sigma'$ that misses at most $k$ vertices of $\sigma'$. It follows that $w$ is $(m-1)$-joinable to the $(l'-k)$-pseudo face $\tau'$ in $L$.

 Hence, by the induction hypothesis for $n-1$, the relative link $L$ of $v$ in $\Delta'$ is $\min\{ \lfloor \frac{l'-k}{k} \rfloor - 1, m-1 -1 \}$-connected. Since $l'\geq l-k$, this completes the proof.
\end{proof}

\section{A Stein--Farley complex for rigid permutation groups}
\label{sec:Stein-Farley-complex}

We will now restrict ourselves to the case of rigid permutation groups $\mathrm{RP}_G(\mathcal{T},\mathcal{B})$ where the graph is a locally finite tree $\mathcal{T}$ on which $G$ acts cocompactly and without inversions of edges. This is equivalent to $G$ being a graph of groups where edge groups have finite index in vertex groups, acting on its Bass--Serre tree. We will also assume that $\mathcal{T}$ is an infinite tree without leaves. 

On a first read it may be helpful for the reader to assume that the set $\mathcal{B}$ is empty, as this is the main motivating example and it might help in developing an intuition for the proof. The reason for which we also cover the case when $\mathcal{B}\subset V(\mathcal{T})$ is non-empty is that it will be required for satisfying the assumptions of \cite[Theorem D]{BelZar-20}, which will provide us with an embedding of $\mathrm{RP}_G(\mathcal{T})$ in a simple group with good finiteness properties.

Following G\'en\'evois, Lonjou and Urech \cite{GenLonUre-22} we will define a Stein--Farley complex for $\mathrm{RP}_G(\mathcal{T},\mathcal{B})$; our exposition follows the one given by  Aramayona, Bux, Flechsig, Petrosyan and Wu in \cite{AraBuxFlePetWu-21}.

We start by defining a suitable poset of equivalence classes $[T,f]$ of pairs of admissible finite subtrees $T\subset \mathcal{T}$ and elements $f\in \mathrm{RP}_G(\mathcal{T},\mathcal{B})$.

Let $\Gamma= G\backslash \mathcal{T}$ be the finite quotient graph. A \emph{system of gates} $\mathfrak{G}$ is a set of half edges of $\Gamma$ (we allow that both half-edges defining an edge are contained in $\mathfrak{G}$). The quotient map $\mathcal{T}\to \Gamma$ allows us to choose a $G$-invariant labelling of the (half-)edges of $\mathcal{T}$ by (half-)edges of $\Gamma$.  In particular, this induces a labelling of some of the half-edges of $\mathcal{T}$ by elements of $\mathfrak{G}$.

A leaf of a finite subtree $T\subset \mathcal{T}$ is called ($\mathfrak{G}$-)\emph{admissible} if the terminal half-edge defining it is labelled by an element $\nu\in \mathfrak{G}$; in this case we will also say that the leaf is of type $\nu$. We call a finite subtree $T\subset \mathcal{T}$  ($\mathfrak{G}$-)\emph{admissible}, if all leaves of $T$ are admissible, and all interior vertices of $T$ have maximal degree (their degree coincides with the degree of the corresponding vertex of $\mathcal{T}$). We call $T$ \emph{admissible with respect to $T_0\subset \mathcal{T}$}, if $T$ is admissible and $T_0\subset \mathcal{T}$ is an admissible tree that is contained in $T$. Usually we will work with admissible trees with respect to a fixed admissible base tree $T_0$ and wherever this is clear from context, we will omit mentioning it. We call $\mathfrak{G}$ \emph{admissible} if every finite subtree of $\mathcal{T}$ is contained in an admissible finite subtree of $\mathcal{T}$.

We note that a system of gates on the graph $\Gamma$ is admissible if and only if any infinite geodesic ray in $\Gamma$ is blocked by some gate. This follows from the local finiteness of the universal cover $T$ and König's Lemma~\cite{Koe-27}. As an immediate consequence of this criterion, we observe that there always exists an admissible gate system on the finite quotient $\Gamma$.  To see this, order the vertices of $\Gamma$ linearly (think of assigning a height to each vertex) and put a gate on each edge in the half-edge pointing to the higher endpoint. For loops, we put gates on both half-edges. For an example of a graph $\Gamma$ equipped with an admissible system of gates obtained this way, see Figure \ref{fig:system-of-gates-for-a-finite-graph}. As $\Gamma$ is finite, there are no infinite strictly descending rays. But as soon as a ray does not descend, it is blocked by a gate.

Figure \ref{fig:admissible-trees} shows examples of admissible and non-admissible trees in the universal cover $\widetilde{\Gamma}$ of the graph $\Gamma$ from Figure \ref{fig:system-of-gates-for-a-finite-graph} with respect to the chosen system of gates.

\begin{figure}
    \centering
    \includegraphics[width=0.5\linewidth]{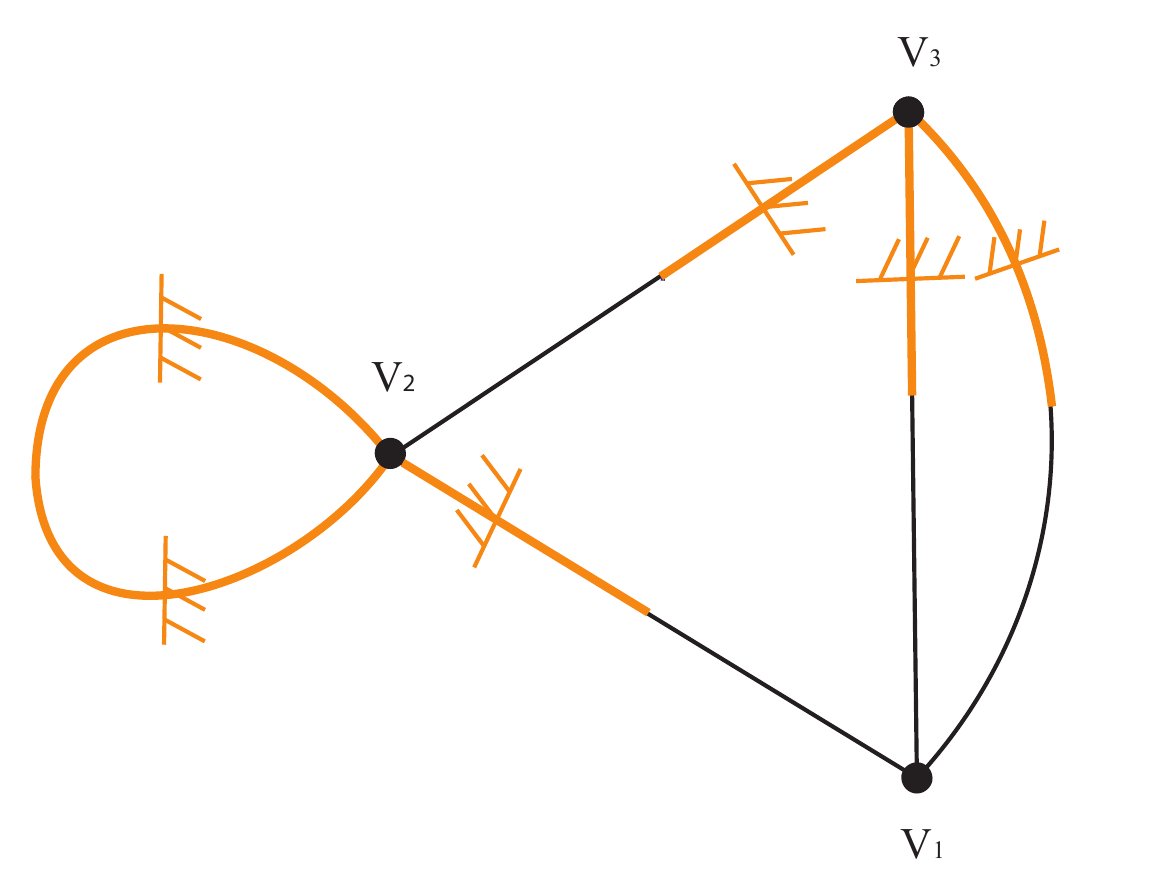}
    \caption{A finite graph $\Gamma$ equipped with an admissible system of gates indicated by the orange half-edges. The walls of the half-spaces illustrate the directions in which the gates are ``closed''.}
    \label{fig:system-of-gates-for-a-finite-graph}
\end{figure}

\begin{figure}
    \centering
    \includegraphics[width=\linewidth]{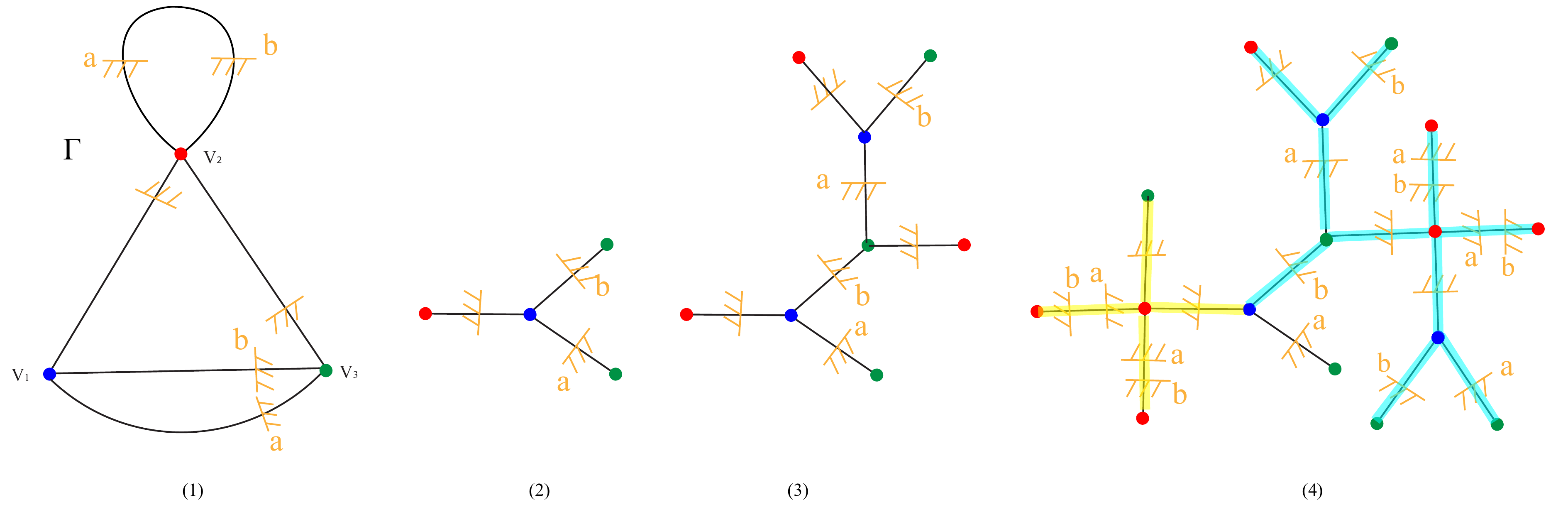}
    \caption{(1) The finite graph $\Gamma$ from Figure \ref{fig:system-of-gates-for-a-finite-graph} equipped with the admissible system of gates $\mathfrak{G}$ indicated in orange; (2) an admissible tree for the universal cover $\mathcal{T}:=\widetilde{\Gamma}$ of the finite graph $\Gamma$ that can serve as base tree $T_0$; (3) a finite subtree of $\mathcal{T}$ which contains $T_0$ and which is not admissible with respect to $\mathfrak{G}$; (4) an admissible tree obtained as an elementary expansion of $T_0$ in two leaves, with the carets we attached to $T_0$ highlighted in yellow and turquoise. Note that in the picture when there are more than one gates between a pair of adjacent vertices, we use letters in $\{a,b\}$ to distinguish them.}
    \label{fig:admissible-trees}
\end{figure}

We now assume that we have chosen an admissible system of gates $\mathfrak{G}$ and consider pairs $(T,f)$ of an admissible tree $T\subset \mathcal{T}$ and an element $f\in \mathrm{RP}_G(\mathcal{T},\mathcal{B})$. We will further assume from now that all our admissible trees are taken with respect to a fixed finite admissible base tree $T_0\subset \mathcal{T}$ which contains the finite set $\mathcal{B}$ in its interior, as we will later apply our results with this assumption. Most of our proofs in this section would work without this extra assumption. However, it is required in the proof of Lemma \ref{lem:intersections-and-unions-of-trees}.

We define an equivalence relation by $(T,f)\sim (S,g)$ if $g^{-1}\circ f\in \mathrm{RP}_G(\mathcal{T},\mathcal{B})$ restricts to a bijection $V(T)\to V(S)$ and has support $\mathrm{int}(T)$ (we only remove the interior of the finite trees $T$ and $S$ to guarantee that the leaves of $T$ are in correspondence with the pieces of $\mathcal{T}\setminus \mathrm{int}(T)$). We denote by $[T,f]$ the equivalence class of $(T,f)$ and by $\mathcal{P}$ the set of equivalence classes. Then $\mathrm{RP}_G(\mathcal{T},\mathcal{B})$ acts on $\mathcal{P}$ by $g\cdot [T,f]:= [T,g\circ f]$.

We introduce a relation on $\mathcal{P}$ by defining that $[T,f]\leq [S,g]$ if $(S,g)\sim (T',f)$ for a tree $T'$ with $T\subseteq T'$. We write $[T,f]< [S,g]$ if the inclusion $T\subset T'$ is strict.

The assumption on the support in the definition of $\sim$ guarantees that admissible trees for equivalent pairs have the same number of leaves. There is thus a well-defined height function $h: \mathcal{P} \to \mathbb{N}$ that maps $[T,f]$ to the number of leaves of $T$, see also Lemma \ref{lem:equivalence-of-pairs}. The $\mathrm{RP}_G(\mathcal{T},\mathcal{B})$-action on $\mathcal{P}$ is height preserving. 

To later construct a cube complex from $\mathcal{P}$, we introduce leaf expansions and elementary expansions. A \emph{leaf expansion} of an admissible tree $T\subset \mathcal{T}$ at a leaf $v$ is an admissible tree $T'\subset \mathcal{T}$ that is minimal such that $T\subsetneq T'$ and $v$ is an interior vertex of $T'$; if the half-edge of $T$ adjacent to $v$ is labelled by $\nu\in \mathfrak{G}$, we also say that $v$ is a \emph{leaf of type $\nu$} and \emph{$T'$ is a leaf expansion of $T$ of type $\nu$}. One can view a leaf expansion at $v$ as attaching a caret to $T$ at $v$; Definition \ref{defn:carets} will make this precise. An \emph{elementary expansion} of an admissible tree $T\subset \mathcal{T}$ is an admissible tree $T'\subset \mathcal{T}$ that is obtained from $T$ by a finite sequence of leaf expansions at leaves of $T$. We write $[T,f]\preccurlyeq [T',f]$, if $T'$ is an elementary expansion of $T$ and $[T,f]\prec [T',f]$, if $T\subsetneq T'$.

We show that the line of argument given in \cite[Section 5]{AraBuxFlePetWu-21} can be adapted to our situation to produce a Stein--Farley complex for $\mathrm{RP}_G(\mathcal{T},\mathcal{B})$ from its action on $\mathcal{P}$.

    \begin{lem}
        \label{lem:intersections-and-unions-of-trees}
        Let $T$ and $T'$ be admissible trees. Then $T\cap T'$ and $T\cup T'$ are admissible trees.
    \end{lem}
    \begin{proof}
        Recall that by our assumptions there is an admissible base tree $T_0$ with $T_0\subset T,~ T'$. 
        
        We first prove that $T\cup T'$ is an admissible tree. Since $T_0\subset T,~ T'$ the union $T\cup T'$ is connected and thus a tree. It clearly contains $T_0$, every interior vertex of $T$ or $T'$ is an interior vertex of $T\cup T'$, and every leaf of $T\cup T'$ is a leaf of $T$ or $T'$, thus admissible. This implies that $T\cup T'$ is admissible.

        Similarly, $T\cap T'$ is a tree containing $T_0$. A vertex of $T\cap T'$ is an interior vertex, if and only if it is an interior vertex of $T$ and $T'$. Thus, all interior vertices of $T\cap T'$ have full degree and all leaves of $T\cap T'$ are admissible leaves of either $T$ or $T'$. This implies that $T\cap T'$ is admissible.
    \end{proof}

\begin{lem}\label{lem:unique-carets}
    For every $\nu\in \mathfrak{G}$ there is a non-trivial unique finite tree $S_{\nu}$  whose edges and vertices are labelled by elements of $\Gamma$ together with a fixed leaf $w_{\nu}$ whose interior half-edge is labelled by $\nu$ with the property that every leaf expansion of type $\nu$ of an admissible tree $T\subset \mathcal{T}$ is obtained by attaching $S_{\nu}$ to $T$ along $w_{\nu}$. Moreover, for every admissible tree $T$ and every choice of leaf $v$ of type $\nu$ there is a unique (non-trivial) leaf expansion $T_v$ of $T$ in $v$ and every admissible expansion of $T$ that has $v$ as an interior vertex contains $T_v$. 
\end{lem}
\begin{proof}
        Let $T$ be an admissible tree and let $v$ be a terminal vertex of a leaf of type $\nu$. Denote by $e$ the edge corresponding to this leaf. Since $\mathcal{T}$ has no leaves, there is at least one further edge adjacent to $v$. This implies that there is a finite subtree $T'$ of $\mathcal{T}$ that contains this further edge and has $T$ as a proper subtree. By our assumptions we may assume that $T'$ is admissible by making it larger if needed. In particular, any such $T'$ will have $v$ as an interior vertex of maximal degree. Thus, the same is true for the intersection of all such $T'$. Since by Lemma \ref{lem:intersections-and-unions-of-trees} this intersection is admissible, there is a unique, non-trivial leaf expansion of $T$ in $v$. We claim that $S_{\nu}:=(T'\setminus T)\cup \overline{e}$ has the desired property, where we label the edge $e$ by $w_{\nu}$.  By construction $S_{\nu}$ is unique and we only need to prove that it does not depend on the choice of $T$ and of one of its leaves labelled by $\nu$. This is an immediate consequence of the fact that $G$ acts transitively on half-edges labelled by $\nu$. Indeed, if $\widehat{T}$ is another admissible tree with a leaf $\widehat{v}$ labelled by $\nu$ and $\widehat{S}_{\nu}$ is the corresponding finite tree with labelled leaf $\widehat{w}_{\nu}$, then there is an element $g\in G$ that maps $w_{\nu}$ to $\widehat{w}_{\nu}$ and such that $\widehat{T}\cup g\cdot S_{\nu}$ is a leaf expansion of $\widehat{T}$ in $\widehat{v}$. By uniqueness $g\cdot S_{\nu}=\widehat{S}_{\nu}$. The moreover part is an immediate consequence of our proof.
    \end{proof}

    \begin{defn}\label{defn:carets}
        For $\nu \in \mathfrak{G}$ we call the tree $S_{\nu}$ from Lemma \ref{lem:unique-carets} \emph{caret (of type $\nu$)} and the leaves of $S_{\nu}$ different from $w_{\nu}$ the \emph{terminal leaves} of the caret $S_{\nu}$.
    \end{defn}
    \begin{rem}
        Lemma \ref{lem:unique-carets} shows that given an admissible tree $T$ and a leaf $v$ of type $\nu$, there is a unique leaf expansion in $v$, which is given by attaching a caret of type $\nu$ to $T$ along the edge $w_{\nu}$ adjacent to $v$.     
    \end{rem}

    \begin{cor}\label{cor:admissible-trees-from-leaf-expansions}
        Let $T\subset \mathcal{T}$ be an admissible tree. Then $T$ is obtained from $T_0$ by a finite sequence $(v_1, \cdots, v_n)$ of leaf expansions in leaves of $\mathcal{T}$. This sequence is unique up to reordering. 
    \end{cor}
    \begin{proof}
        We inductively construct a sequence of trees $T_0, T_1, \cdots, T_n=T$ such that $T_i$ is obtained from $T_{i-1}$ by a leaf expansion as follows. If $T_0\subsetneq T$, then there is a leaf $v_1$ of $T_0$, which is an interior vertex of $T$. Lemma \ref{lem:unique-carets} implies that the leaf expansion $T_1$ of $T_0$ in $v_1$ is contained in $T$. Since $T$ is finite, repeating this argument finitely many times yields the existence of such a sequence. For uniqueness we observe that the $v_i$ correspond precisely to the interior vertices of $T$, which are not interior vertices of $T_0$ and which are adjacent to a half-edge labelled by an element of $\mathfrak{G}$ that points away from $T_0$.
    \end{proof}

\begin{lem}\label{lem:basic-properties-of-expansions}
    Suppose $(T_1,f_1)\sim (T_2,f_2)$. Then:
    \begin{enumerate}
        \item If $T_1'$ is an admissible tree containing $T_1$, then there is an admissible tree $T_2'$ containing $T_2$ such that $[T_1',f_1]=[T_2',f_2]$.
        \item If $[T_1',f_1]=[T_2',f_2]$ and $T_1'$ is obtained from $T_1$ by an elementary expansion (of $\ell$ leaves), then $T_2'$ is obtained from $T_2$ by an elementary expansion (of $\ell$ leaves).
        \item If $f_1=f_2$, then $T_1=T_2$.
    \end{enumerate}
\end{lem}
\begin{proof}
    By definition, $f_2^{-1}\circ f_1|_{\mathcal{T}\setminus {\rm int}(T_1)}$ maps the connected components of $\mathcal{T}\setminus {\rm int}(T_1)$ rigidly to the connected components of $\mathcal{T}\setminus {\rm int}(T_2)$. In particular, it preserves leaf types and maps leaf expansions to leaf expansions of the same type. Thus, setting $T_2':= (f_2^{-1}\circ f_1)(T_1')$ shows (1) and (2).

    For (3) observe that by definition of the equivalence relation the identitiy map $id_{\mathcal{T}}=f_1^{-1}\circ f_2$ maps $T_1$ to $T_2$ bijectively. Thus, $T_1=T_2$.
\end{proof}

\begin{lem}\label{lem:uniqueness-for-equal-maps}
    Assume that $[T,f]\leq v\leq w\leq [T',f]$. Then there are unique admissible trees $T_v$ and $T_w$ with $T\subset T_v\subset T_w \subset T'$, $v=[T_v,f]$ and $w=[T_w,f]$.
\end{lem}
\begin{proof}
    The uniqueness is an immediate consequence of Lemma \ref{lem:basic-properties-of-expansions}(3). We now prove the existence. By definition there are representatives $[T,f]=[S,g]$, $v=[S_v,g]$ such that $S\subset S_v$. Lemma \ref{lem:basic-properties-of-expansions}(1) implies that there is an admissible tree $T_v$ with $T\subset T_v$ such that $v=[T_v,f]$. Again using the definition of equivalence, there are representatives $[R_v,h]=[T_v,f]=v$ and $[R',h]=[T',f]$ such that $R_v\subset R'$. Thus, by Lemma \ref{lem:basic-properties-of-expansions}(1), there is a tree $T_v'\subset T'$ with $[T_v',f]=v$. Lemma \ref{lem:basic-properties-of-expansions} implies that $T_v'=T_v$. This proves that there is a unique tree $T_v$ such that $v=[T_v,f]$ and $T\subset T_v\subset T'$. Applying the same line of argument to $v=[T_v,f]\leq w \leq [T',f]$ completes the proof.
\end{proof}

We denote by $|\mathcal{P}|$ the geometric realisation of the poset $\mathcal{P}$. It is the simplicial complex, whose $n$-simplices are the subsets $\left\{v_0,\cdots, v_n\right\}\subset \mathcal{P}$ with $v_0<v_1<\cdots <v_n$. 

\begin{lem}\label{lem:directed-poset}
    $(\mathcal{P},\leq)$ is a directed poset. In particular, its geometric realisation $|\mathcal{P}|$ is contractible.
\end{lem}
\begin{proof}
    We start by showing that $(\mathcal{P},\leq)$ is a poset. Let $u=[T_u,f_u],~v=[T_v,f_v],~w=[T_w,f_w]\in \mathcal{P}$. Clearly $u\leq u$. If $u\leq v$ and $v\leq u$, then Lemma \ref{lem:uniqueness-for-equal-maps} implies that $v=[T',f_u]$ for some unique tree $T'$ with $T_u\subseteq T'\subseteq T_u$. Thus, $u=v$. If $u\leq v$ and $v\leq w$, by Lemma \ref{lem:basic-properties-of-expansions}(1) we can first choose a representative for $v$ of the form $v=[T',f_u]$ with $T_u\subset T'$. Another application of Lemma \ref{lem:basic-properties-of-expansions}(1) shows that $w=[T'',f_u]$ for some $T''$ with $T'\subseteq T''$. Thus, $u\leq w$. This proves that $(\mathcal{P},\leq)$ is a poset.
    
    To see that $(\mathcal{P},\leq)$ is directed, we need to show that for all $u=[T_u,f_u],~v=[T_v,f_v]\in \mathcal{P}$ there is an element $w\in \mathcal{P}$ with $u,~v\leq w$. Since every finite subtree of $\mathcal{T}$ is contained in an admissible finite subtree of $\mathcal{T}$, we can choose an admissible tree $T'_u\subset \mathcal{T}$, which contains $T_u \cup (f_u^{-1}\circ f_v)(T_v)$ and is a support for $f_v^{-1}\circ f_u\in \mathrm{RP}_G(\mathcal{T},\mathcal{B})$. Then $T'_v:=(f_v^{-1}\circ f_u)(T'_u)\supset T_v$ is a support for $f_u^{-1}\circ f_v$. Thus, we can choose $w=[T_u',f_u]=[T_v',f_v]$, proving that $\leq$ is directed.

    Finally, $|\mathcal{P}|$ is contractible, since the geometric realisation of a directed poset is contractible.
\end{proof}

We say that an $n$-simplex $v_0<\cdots <v_n$ of $|\mathcal{P}|$ is \emph{elementary} if $v_0\preccurlyeq v_n$. In this case we clearly also have $v_i\prec v_j$ for every $i<j$.  We will now restrict to the full subcomplex $\mathfrak{X}$ of $|\mathcal{P}|$ consisting of elementary simplices. We call $\mathfrak{X}$ the \emph{Stein complex} associated with $\mathcal{P}$ with respect to $T_0$. The action of $\mathrm{RP}_G(\mathcal{T},\mathcal{B})$ on $\mathcal{P}$ preserves the relation $\preccurlyeq$. Thus, it restricts to an action on $\mathfrak{X}$. We will now prove that $\mathfrak{X}$ is also contractible. For $u\leq v$ we define the interval $[x,y]:=\left\{v\in \mathcal{P}\mid x\leq v\leq y\right\}$. Similarly, we define the intervals $(x,y]$, $[x,y)$ and $(x,y)$.

\begin{lem}\label{lem:intervals-are-contractible}
    Let $x,~y\in \mathcal{P}$ with $x< y$ and $x\not \prec y$. Then the geometric realisations $|(x,y)|$, $|(x,y]|$, $|[x,y)|$ and $|[x,y]|$ are contractible.
\end{lem}
\begin{proof}
    We give a proof for $|(x,y)|$. The other cases are analogous. By definition, there exist $f\in \mathrm{RP}_G(\mathcal{T},\mathcal{B})$ and admissible trees $T\subset T'$ such that $x=[T,f]$ and $y=[T',f]$. By Lemma \ref{lem:uniqueness-for-equal-maps}, for $v,w\in \left[x,y\right]$ with $v\leq w$, there exist unique trees $T_v$ and $T_w$ with $v=[T_v,f]$, $w=[T_w,f]$ and $T\subset T_v\subset T_w\subset T'$. In particular, for every $v\in \left[x,y\right]$ there is a unique maximal elementary expansion $v_0=[T_{v,m},f]$ of $x$ with $v_0\leq v$. Thus, we can define a map of posets $\phi: [x,y]\to [x,y]$, $v\mapsto v_0$. Moreover, since $x\not \prec y$, $y_0<y$. Since also $x<v$ implies $x<v_0$, the restriction $\phi: (x,y)\to (x,y)$ is well-defined and satisfies $v\geq \phi(v)\leq y_0$ for all $v\in (x,y)$. Thus, \cite[Section 1.5]{Qui-78} implies that there is a well-defined conical contraction of $|(x,y)|$ onto $y_0$. In particular, $|(x,y)|$ is contractible.
\end{proof}

\begin{prop}\label{prop:stei-cpx-contractible}
    The complex $\mathfrak{X}$ is contractible.
\end{prop}
    \begin{proof}
        Since $|\mathcal{P}|$ is contractible by Lemma \ref{lem:directed-poset}, it suffices to show that $|\mathcal{P}|$ can be constructed inductively from $\mathfrak{X}$ without changing the homotopy type. For a closed interval $[x,y]$ in $\mathcal{P}$, we define $r([x,y]):=h(y)-h(x)$, where we recall that $h$ maps a vertex $[T,f]$ to the number of leaves of $T$. We inductively attach the contractible complexes $|[x,y]|$ to $\mathfrak{X}$ in increasing order of $r$-value. The subcomplex $|[x,y]|$ is attached along the complex $|(x,y]|\cup |[x,y)|$. The latter is contractible, since by Lemma \ref{lem:intervals-are-contractible} it is the union of two contractible complexes whose intersection $|(x,y)|$ is contractible. This shows that attaching the complexes $|[x,y]|$ up to a given $r$-value does not change the homotopy type of $\mathfrak{X}$. Thus, $|\mathcal{P}|$ is homotopy equivalent to $\mathfrak{X}$, since it is the result of attaching all intervals $|[x,y]|$ to $\mathfrak{X}$.
    \end{proof}

    We now show that $\mathfrak{X}$ admits the structure of a cube complex, where the cubes of dimension $n$ are the geometric realisations of the intervals $[x,y]$ with $x\preccurlyeq y$ and $n=r([x,y])$. We start by showing that the geometric realisations of such intervals can canonically be identified with cubes. This is an immediate consequence of:

    \begin{lem}\label{lem:boolean}
        If $x\preccurlyeq y$, then $[x,y]$ is a finite Boolean lattice.
    \end{lem}
    \begin{proof}
        Let $x=[T,f]$ and let $y=[T',f]$, where $T'$ is an elementary expansion of $T$. Let $z\in [x,y]$. By Lemma \ref{lem:uniqueness-for-equal-maps} there is a unique elementary expansion $T_z$ with $z=[T_z,f]$ and $T\subset T_z \subset T'$. It follows that $[x,y]$ is a Boolean lattice on the set of leaf expansions needed to obtain $T'$ as an elementary expansion of $T$.
    \end{proof}

    To prove that these cubes equip $\mathfrak{X}$ with the structure of a cube complex, we need to show that any two cubes intersect in a cube. This will require a preliminary result.

    \begin{lem}\label{lem:intersections-of-cubes}
        Assume $x\preccurlyeq y$ and $z\preccurlyeq w$. Denote $\mathcal{S}=[x,y]\cap[z,w]$. Then for any $p,q\in \mathcal{S}$ there are $s,t\in \mathcal{S}$ such that $s\leq p,q\leq t$.
    \end{lem}
    \begin{proof}
        The conclusion is trivial if $p=q$, so assume $p\prec q$. By our assumptions there are representatives of the form $x=[T,f_1]$, $y=[\overline{T},f_1]$, $z=[S,f_2]$, $w=[\overline{S},f_2]$ such that $\overline{T}$ (resp. $\overline{S}$) is an elementary expansion of $T$ (resp. $S$). By Lemma \ref{lem:uniqueness-for-equal-maps} there are thus elementary expansions $S_p$ and $S_q$ of $S$, as well as elementary expansions $T_p$ and $T_q$ of $T$ such that
        \[
            p=[T_p,f_1]=[S_p,f_2] \mbox{ and } q=[T_q,f_1]=[S_q,f_2].
        \]

        Then the restrictions 
        \[
        f_2^{-1}\circ f_1|_{\mathcal{T}\setminus \mathrm{int}(T_p)}: \mathcal{T}\setminus \mathrm{int}(T_p) \to \mathcal{T}\setminus \mathrm{int}(S_p)
        \]
        and
        \[
        f_2^{-1}\circ f_1|_{\mathcal{T}\setminus \mathrm{int}(T_q)}: \mathcal{T}\setminus \mathrm{int}(T_q) \to \mathcal{T}\setminus \mathrm{int}(S_q)
        \]        
        are well-defined rigid bijections with $(f_2^{-1}\circ f_1)(T_p)=S_p$, $(f_2^{-1}\circ f_1)(T_q)=S_q$.

        We deduce that $(f_2^{-1}\circ f_1)(T_p\cup T_q)=S_p\cup S_q$, $(f_2^{-1}\circ f_1)(T_p\cap T_q)=S_p\cap S_q$ and the restriction of $f_2^{-1}\circ f_1$ to $\mathcal{T}\setminus \mathrm{int}(T_p\cup T_q)$ is rigid. By Lemma \ref{lem:intersections-and-unions-of-trees} $T_p\cap T_q$, $T_p\cup T_q$, $S_p\cap T_q$, and $S_p\cup S_q$ are admissible and it is easy to see that they are elementary expansions of $T$, resp. $S$. We can define the elements $s=[T_p\cap T_q,f_1]$ and $t=[T_p\cup T_q,f_1]$. By construction they satisfy $s\leq p,~q\leq t$ and $s,t\in [x,y]$. Since, also by construction, $s=[T_p\cap T_q,f_1]=[S_p\cap S_q,f_2]$ and $t=[T_p\cup T_q,f_1]=[S_p\cup S_q,f_2]$, we also have $s,t\in [z,w]$. Thus, $s,t\in \mathcal{S}$, completing the proof.
    \end{proof}

    \begin{prop}
        We can equip $\mathfrak{X}$ with the structure of a cube complex, where the cubes are defined by the Boolean lattices on the intervals $[x,y]$, where $x\preccurlyeq y$.
    \end{prop}
    \begin{proof}
        It suffices to show that for any two cubes $[x,y]$ and $[z,w]$ the intersection $\mathcal{S}=[x,y]\cap [z,w]$ is a cube. We assume that $\mathcal{S}\neq \emptyset$. Since $\mathcal{S}$ is finite, Lemma \ref{lem:intersections-of-cubes} implies that there are $s,t\in \mathcal{S}$ with $\mathcal{S}\subseteq [s,t]$. Since $[s,t]$ is a subinterval of both $[x,y]$ and $[z,w]$, we have $\mathcal{S}=[s,t]$. Finally, since $x\preccurlyeq s\leq t \preccurlyeq y$, we have $s\preccurlyeq t$, implying that $[x,y]$ and $[z,w]$ intersect in the common subcube $[s,t]$.
    \end{proof}

    \begin{defn}
        We will call the complex $\mathfrak{X}$ equipped with the above cubical structure, the \emph{Stein--Farley complex} associated to the rigid permutation group $\mathrm{RP}_G(\mathcal{T},\mathcal{B})$ with respect to $T_0$.
    \end{defn}

    \begin{rem}
        The only way in which $\mathfrak{X}$ depends on the finite set $\mathcal{B}$ is via the assumption that $\mathcal{B}\subset T_0$. In particular, if $\mathcal{B}\subset \mathcal{B}'$, then we can choose finite trees $T_0\subset T_0'$ with $\mathcal{B}\subset V(T_0)$ and $\mathcal{B}'\subset V(T_0')$. With these choices the Stein--Farley complex associated to $\mathrm{RP}_G(\mathcal{T},\mathcal{B}')$ with respect to $T'_0$ is naturally a subcomplex of the Stein--Farley complex associated to $\mathrm{RP}_G(\mathcal{T},\mathcal{B})$ with respect to $T_0$.
    \end{rem}

\section{Leaf expansions and the history of trees}
\label{sec:expansions-of-trees}

In this section we discuss the sequence of leaf expansions defining a tree, its number of leaves of each type and the equivalence of pairs $(T,f)$. 

We identify the finite set of gates $\mathfrak{G}$ with $\left\{1,\cdots, k\right\}$. By Corollary \ref{cor:admissible-trees-from-leaf-expansions} for every admissible tree $T\subset \mathcal{T}$ there is an up to reordering unique sequence $(v_1,\cdots, v_n)$ of vertices of $T$ such that $T$ is obtained from $T_0$ by inductively expanding these vertices. For $i\in \left\{1,\cdots, k\right\}$ we denote by $N_i(T)$ the number of leaf expansions of type $i$ used to obtain $T$ from $T_0$ and by $L_i(T)$ the number of admissible leaves of $T$ of type $i$; we call the $k$-tuple $(N_1(T),\dots, N_k(T))$ the \emph{history of $T$}. Denote, moreover, by $M_{ij}$ the number of terminal leaves of type $i$ of the caret $S_j$ (see Definition 
\ref{defn:carets}). This defines a matrix $M=\left(M_{ij}\right)\in \mathbb{N}_0^{k\times k}$. Let further $I(T)$ be the number of interior vertices of $T$, $I_j$ be the number of interior vertices of $S_j$ and $I_0$ be the number of interior vertices of $T_0$.

\begin{lem}\label{lem:linear-algebra-determines-counts}
    Let $T\subset \mathcal{T}$ be an admissible tree. Then
    \[
        \left(\begin{array}{c}L_1(T)\\ \vdots \\ L_k(T)\end{array}\right)= \left(\begin{array}{c}L_1(T_0)\\ \vdots \\ L_k(T_0)\end{array}\right)+ (M - Id)\left(\begin{array}{c}N_1(T)\\ \vdots \\ N_k(T)\end{array}\right)
    \]
    and
    \[
        I(T) = I_0 + \left( I_1,  \cdots , I_k \right) \left(\begin{array}{c}N_1(T)\\ \vdots \\ N_k(T)\end{array}\right).
    \]
\end{lem}

\begin{proof}
    By Lemma \ref{lem:unique-carets} a leaf expansion of type $i$ of a tree is given by attaching the tree $S_i$ to a leaf of type $i$ along $w_i$. Thus, such an expansion increases the total leaf count by
    \[
        \left(\begin{array}{c} M_{1i}\\ \vdots \\ M_{ki}\end{array}\right) - e_i = \left(M-Id\right)e_i,
    \]
    where $e_i$ denotes the $i$-th standard basis vector and the total number of interior vertices increases by $I_i$. This proves the assertion.
\end{proof}

\begin{lem}\label{lem:equivalence-of-pairs}
    Let $T_1,~ T_2\subset \mathcal{T}$ be admissible trees and let $f_1\in \mathrm{RP}_G(\mathcal{T},\mathcal{B})$. There is an element $f_2\in \mathrm{RP}_G(\mathcal{T},\mathcal{B})$ with $[T_1,f_1]=[T_2,f_2]$ if and only if $I(T_1)=I(T_2)$ and $L_i(T_1)=L_i(T_2)$ for $1\leq i \leq k$.
\end{lem}
\begin{proof}
    After left-multiplying by $f_1^{-1}$ if needed, we may assume that $f_1=id_{\mathcal{T}}$. Then $[T_1,id_{\mathcal{T}}]=[T_2,f_2]$ if and only if $f_2$ defines a bijection $T_2\to T_1$ with $f_2|_{\mathcal{B}}=id_{\mathcal{B}}$ and a rigid map $f_2: \mathcal{T}\setminus {\rm int}(T_2)\to \mathcal{T}\setminus {\rm int}(T_1)$.  In particular, $f_2$ bijectively maps interior vertices of $T_2$ to interior vertices of $T_1$ and leaves of $T_2$ of type $i$ to leaves of $T_1$ of type $i$. This implies that $I(T_1)=I(T_2)$ and $L_i(T_1)=L_i(T_2)$ for $1\leq i \leq k$. 

    Conversely, if $I(T_1)=I(T_2)$ and $L_i(T_1)=L_i(T_2)$ for $1\leq i \leq k$, we define $f_2$ with $[T_1,id_{\mathcal{T}}]=[T_2,f_2]$ as follows. We first choose a bijection $\phi: T_2\to T_1$ with $\phi|_{\mathcal{B}}=id_{\mathcal{B}}$ that maps interior vertices to interior vertices and leaves of type $i$ to leaves of type $i$. Since by our assumptions the action of $G$ on half-edges labelled by $i$ is transitive, for every leaf $v$ of $T_2$ of type $i$ there is an element $g_v\in G$ with $g_v\cdot v= \phi(v)$. There is thus an element $f_2\in \mathrm{RP}_G(\mathcal{T},\mathcal{B})$ that coincides with $\phi$ on $T_2$ and with $g_v$ on the connected component of $\mathcal{T}\setminus {\rm int}(T_2)$ based at $v$. By construction $f_2$ has support $\mathrm{int}(T_2)$ and thus $[T_1,id_{\mathcal{T}}]=[T_2,f_2]$.
\end{proof}

Lemma \ref{lem:equivalence-of-pairs} shows that the number of leaves of each type of a tree representing a vertex $x\in \mathfrak{X}$ does not depend on the choice of representative. Thus, we will subsequently refer to the \emph{set of leaves of $x$} as the set of leaves of any tree $T$ with $x=[T,f]$; it is well-defined up to the identification given by $\sim$. Moreover, if $T_1$ and $T_2$ are trees with $I(T_1)=I(T_2)$ and $L_i(T_1)=L_i(T_2)$ for $1\leq i \leq k$ we will say that $T_1$ can be \emph{rearranged} to $T_2$.

\begin{defn}
    We say that a rigid permutation group $\mathrm{RP}_G(\mathcal{T},\mathcal{B})$ together with an admissible set of gates $\mathfrak{G}$ and an admissible base tree $T_0\subset \mathcal{T}$ has the \emph{viral expansion property} if $L_i(T_0)\geq 2$ and $M_{ii}\geq 3$ for every $i\in \left\{1,\cdots, k\right\}$. 
\end{defn}

\begin{rem}\label{rem:viral-expansion-property}
    If $M_{ii}\geq 3$ for all $i\in \left\{1,\cdots, k\right\}$ we can always assume that $L_i(T_0)\geq 2$  after possibly replacing $\mathfrak{G}$ by a subset and $T_0$ by a larger admissible tree. To see this, we first observe that we may assume that every element from $\mathfrak{G}$ appears as a leaf of some admissible tree, after removing elements that do not if needed. Under this assumption, we can then replace $T_0$ by a tree which has at least one leaf of every type by making it larger if needed. Finally, since $M_{ii}\geq 3$, we can then perform a leaf expansion of type $i$ if needed to obtain an admissible base tree which has at least two leaves of every type.  
\end{rem}

\begin{lem}
    If $(\mathrm{RP}_G(\mathcal{T},\mathcal{B}),\mathfrak{G}, T_0)$ has the viral expansion property, then every element $(N_1,\dots,N_k)\in \mathbb{N}_0^k$ can be realised as history of some admissible tree $T$. If, moreover, $T'$ is another tree with this property and $f\in \mathrm{RP}_G(\mathcal{T},\mathcal{B})$, then there is an element $f'\in \mathrm{RP}_G(\mathcal{T},\mathcal{B})$ with $[T,f]=[T',f']$. 
\end{lem}
\begin{proof}
    Since by assumption $T_0$ has a leaf of every type and every leaf expansion of type $i$ increases the number of leaves of type $i$, we can perform a sequence of $n=\sum_{i=1}^k N_i$ leaf expansions $(v_1,\cdots, v_n)$ consisting of $N_i$ expansions of type $i$ in any order we want. This shows the existence of a tree $T$ with the desired properties. The second assertion is an immediate consequence of Lemmas \ref{lem:linear-algebra-determines-counts} and \ref{lem:equivalence-of-pairs}.
\end{proof}

\section{Bestvina--Brady Morse theory}
\label{sec:BB-Morse-theory}
To prove that $\mathrm{RP}_{\mathcal{T}}(G)$ has good finiteness properties, we will use Brown's criterion and Bestvina--Brady Morse theory on affine cell complexes. We start by briefly introducing affine cell complexes, for details refer to \cite[Section 2]{BesBra-97}. We will follow \cite{GenLonUre-22} and \cite{AraBuxFlePetWu-21} in our exposition.

An \emph{affine cell complex} is a cell complex such that every cell $e$ is identified with a convex polyhedral cell $C_e$ in some $\mathbb{R}^n$ (where $n$ may depend on the cell) via a characteristic map $\chi : C_e \to e$ and the restriction of $\chi$ to any face of $C_e$ is the characteristic function of another cell, up to a precomposition by an affine homeomorphism. One should think of an affine cell complex, as a cell complex obtained by gluing convex polyhedral cells along their faces and for us the most important example will be cube complexes.

A \emph{height function} (or \emph{Morse function}) on an affine cell complex $X$ is a map $h: X\to \mathbb{R}$ such that the restriction $h|_{e}$ to every cell $e$ of $X$ is affine linear and constant only if $\mathrm{dim}(e)=0$, and the image of the $0$-skeleton in $\mathbb{R}$ is discrete. For $a\in \mathbb{R}$, the \emph{sublevel set $X_a$ of $X$ with respect to $h$} is the full subcomplex of $X$ spanned by the vertices $x\in X$ with $h(x)\leq a$. 

The \emph{link ${\rm Lk}_X(v)$ of a vertex $v\in X$} of a locally finite affine cell complex $X$ is the cell complex defined by intersecting $X$ with a sufficiently small sphere around $v$: the vertices of ${\rm Lk}_X(v)$ are in correspondence with the edges of $X$ based at $v$ and a collection of vertices spans a cell if they correspond to the set of edges of a cell of $X$ at a vertex that is identified with $v$ (a more precise definition can be found in \cite{BesBra-97}). 

The \emph{ascending/descending link} ${\rm Lk}^{\uparrow/\downarrow}_X(v)$ of $v$ (with respect to $h$) is the full subcomplex of ${\rm Lk}_X(v)$ spanned by all vertices corresponding to edges along which $h$ is increasing/decreasing. 

To prove that our groups have good finiteness properties, we will use the following well-known criterion, obtained by combining Brown's criterion \cite{Bro-87} with Bestvina--Brady Morse theory \cite{BesBra-97} (see e.g. \cite[Proposition 5.15]{GenLonUre-22}).

\begin{thm}\label{thm:Brown-and-BB}
    Let $G$ be a group and let $X$ be a contractible affine $G$-CW complex. Assume that there is $n\in \mathbb{N}_0$ such that all cell stabilisers of $k$-cells are $F_{n-k}$. Let $h:X\to \mathbb{R}$ be a $G$-invariant height function. Assume that for every $a\in \mathbb{R}$ the sublevel set $X_a$ is $G$-cocompact and that, for every $m\in \left\{0,\cdots, n-1\right\}$, there exists some $r\in \mathbb{R}$ such that the descending link of every vertex $x\in X$ with $h(x)\geq r$ is $m$-connected. Then $G$ is of type $F_n$.     
\end{thm} 

We will apply this to the Stein--Farley complex $\mathfrak{X}$ equipped with the height function $h:\mathfrak{X}\to \mathbb{R}$ defined by mapping a vertex $[T,f]$ to the number of leaves of $T$ and extending affine linearly on cubes. In particular, $h$ only depends on $T$ and it is easy to see that this is indeed a well-defined height function for the $\mathrm{RP}_G(\mathcal{T},\mathcal{B})$-action on $\mathfrak{X}$. Moreover, since for every $a\in \mathbb{N}$ there are only finitely many admissible trees with $\leq a$ leaves, the $\mathrm{RP}_G(\mathcal{T},\mathcal{B})$-action on sublevel sets is cocompact. Since we have also already checked that $\mathfrak{X}$ is contractible to deduce from Theorem \ref{thm:Brown-and-BB} that $\mathrm{RP}_G(\mathcal{T},\mathcal{B})$ is of type $F_n$ for some $n\in \mathbb{N}$ it suffices to check the following two conditions:
\begin{enumerate}
    \item \label{itm:conn-desc-links}For every $m\in \mathbb{N}$ there is an $r\in \mathbb{N}$ such that the descending link of every vertex of height at least $r$ is $m$-connected.
    \item \label{itm:fin-props-of-stabilisers} The stabiliser in $\mathrm{RP}_G(\mathcal{T},\mathcal{B})$ of every cell of $\mathfrak{X}$ is of type $F_n$.
\end{enumerate}
We will do this in Sections \ref{sec:descending-links} and \ref{sec:finiteness-properties-of-cell-stabilisers}.

\section{Analysis of the descending links}
\label{sec:descending-links}

We now assume that we are given a triple $(\mathrm{RP}_G(\mathcal{T},\mathcal{B}),\mathfrak{G},T_0)$ that satisfies the viral expansion property. As before we identify $\mathfrak{G}$ with the set $\left\{1,\cdots, k\right\}$. In particular, there are $k$ caret types. 

In this section we will show that under these assumptions for every $m\in \mathbb{N}$ there is an $r\in \mathbb{N}$ such that the descending link of every vertex of $\mathfrak{X}$ of height at least $r$ is $m$-connected. To this end we will first describe the descending links of $\mathfrak{X}$ in Section \ref{subsec:description-of-descending-links} and then analyse their connectivity properties in Section \ref{subsec:connectivity-descending-links}.

\subsection{Descending links in $\mathfrak{X}$}
\label{subsec:description-of-descending-links}

Recall that the descending link ${\rm Lk}^{\downarrow}(x)$ of $x=[T,f]\in \mathfrak{X}$ is the full subcomplex of the link ${\rm Lk}(x)$ of $x$ spanned by the vertices that correspond to edges of $\mathfrak{X}$ on which $h$ is decreasing with respect to $h(x)$. In particular, for $\ell\geq 0$ the $\ell$-simplices in ${\rm Lk}^{\downarrow}(x)$ are in correspondence with the $(\ell+1)$-cubes $C$ of $\mathfrak{X}$ which have $x$ as a vertex in which $h|_C$ attains its maximum. In this case $C=[y,x]$ for a suitable $y\in \mathfrak{X}$  with $y\prec x$.

\begin{lem}\label{lem:l-simplices-in-descending-link}
    Let $x\in \mathfrak{X}$. The $\ell$-simplices in ${\rm Lk}^{\downarrow}(x)$ are in one-to-one correspondence with the elements $y\in \mathfrak{X}$ which admit a representative of the following form:
    
    There is an admissible tree $T$ and an element $f\in \mathrm{RP}_G(\mathcal{T},\mathcal{B})$ such that $y=[T,f]$, $x=[T',f]$ and $T'$ is obtained from $T$ by an elementary expansion that consists of precisely $(\ell+1)$ leaf expansions of $T$. Moreover, the vertices of the $\ell$-simplex are in correspondence with the elements of the collection of carets we need to attach to $T$ to obtain $T'$.
\end{lem}
\begin{proof}
    Let $\rho$ be an $\ell$-simplex of ${\rm Lk}^{\downarrow}(x)$ and let $C=[y,x]$ be the $(\ell+1)$-cube of $\mathfrak{X}$ representing $\rho$. Then $y\preccurlyeq x$, meaning that there are representatives $y=[T,f]$ and $x=[T',f]$ such that $T'$ is obtained from $T$ by expanding finitely many of its leaves, or, equivalently, $\overline{T'\setminus {\rm int}(T)}$ consists of a collection of finitely many disjoint carets. The condition that $[y,x]$ is $(\ell+1)$-dimensional implies that the number of disjoint carets is $\ell+1$ (note that by Lemma \ref{lem:basic-properties-of-expansions}(2) this number does not depend on the choice of representatives). 
    
    Conversely, it is clear that for any $y=[T,f]$ such that $x=[T',f]$, where $T'$ is obtained from $T$ by an elementary expansion that consists of precisely $(\ell+1)$ leaf expansions of $T$, the interval $[y,x]$ defines an $(\ell+1)$-cube of $\mathfrak{X}$ such that $h|_{[y,x]}$ attains its maximum in $x$ and thus an $\ell$-simplex of ${\rm Lk}^{\downarrow}(x)$.

    Finally, the moreover-part is an immediate consequence of this description of the $\ell$-simplices.
\end{proof}
As a consequence we see that the descending link of $x$ is a subcomplex of a certain matching complex on the vertex set of ${\rm Lk}^{\downarrow}(x)$ in the following sense.
\begin{cor}\label{cor:matching-complex}
    Let $x\in \mathfrak{X}$ and let $\rho$ be an $\ell$-simplex of ${\rm Lk}^{\downarrow}(x)$. Then $\rho$ determines a set of $\ell+1$ pairwise disjoint subsets of the set of leaves of $x$, each of which is in bijection with a set of terminal leaves of a caret.
\end{cor}
\begin{proof}
    This is an immediate consequence of Lemma \ref{lem:l-simplices-in-descending-link} and the definition of the set of leaves of a vertex of $\mathfrak{X}$.
\end{proof}

\subsection{Connectivity of descending links} \label{subsec:connectivity-descending-links}
We will now use the description of the descending links from Section \ref{subsec:description-of-descending-links} to show that they satisfy the desired connectivity properties. We recall that an (admissible) tree $S$ can be \emph{rearranged} to an (admissible) tree $T$ if $I(S)=I(T)$ and $L_i(S)=L_i(T)$ for $1\leq i \leq k$. We will further say for a finite collection $\rho$ of carets (where we allow several carets to have the same type) that an admissible tree $T$ \emph{is an elementary $\rho$-expansion}, if there is an admissible subtree $T'$ such that $T$ is obtained from $T'$ by an elementary expansion that consists of attaching the carets in $\rho$ to leaves of $T'$.

\begin{lem}\label{lem:maximum-in-alternative-histories}
Let $\rho$ be a finite collection of carets and let $i\in \left\{1,\cdots, k\right\}$. Then there is a constant $\alpha(\rho,i)\in\mathbb{N}_0$ with the following property:

If a tree with history $(N_1,\cdots, N_k)$ can be rearranged to a tree which is an elementary $\rho$-expansion, then there is a subhistory $(M_1,\cdots,M_k)\leq (N_1,\cdots, N_k)$ with $M_i\leq \alpha(\rho,i)$ such that any tree with this subhistory can be rearranged to a tree which is an elementary $\rho$-expansion.
\end{lem}

\begin{proof}
 The proof is by contradiction. Assume there is a collection $\rho$ of carets  and an $i\in \left\{1,\cdots,k\right\}$ such that for all $\alpha>0$ there is a tree $T$ with history $(N_1,\cdots,N_k)$ satisfying the following properties:
 \begin{enumerate}
     \item $T$ can be rearranged to a tree which is an elementary $\rho$-expansion;
     \item $N_i\geq \alpha$;
     \item no tree corresponding to a proper subhistory $(M_1,\cdots, M_k)$ of $(N_1,\cdots, N_k)$ with $M_i<N_i$ can be rearranged to a tree which is an elementary $\rho$-expansion.
 \end{enumerate}

 Without loss of generality we may assume that $i=k$. Observe that our assumptions then imply that there is a sequence of admissible trees $T_j$ with histories $(N_{1,j},\cdots, N_{k,j})$ such that
 \begin{enumerate}
   \item
     the sequence $(N_{k,j})_j$ is strictly increasing,
   \item
     each tree $T_j$ can be rearranged to a tree which is an elementary $\rho$-expansion and
   \item
     there is no proper subhistory $(M_{1,j},\cdots, M_{k,j})<(N_{1,j},\cdots, N_{k,j})$ with $M_{k,j}<N_{k,j}$ for which a corresponding tree can be rearranged to a tree that is an elementary $\rho$-expansion.
 \end{enumerate}
 Now consider two indices $j_1 < j_2$. If we had $(N_{1,j_1},\cdots, N_{k-1,j_1})\leq (N_{1,j_2},\cdots, N_{k-1,j_2})$ coordinatewise, then $(N_{1,j_1},\cdots, N_{k,j_1})\leq (N_{1,j_2},\cdots, N_{k,j_2})$ is a proper subhistory with $N_{k,j_1}<N_{k,j_2}$ as $(N_{k,j})_j$ is strictly increasing. This contradicts the third item as $T_{j_1}$ can be rearranged to a tree that is an elementary $\rho$-expansion.

 Consequently, for every $j\geq 1$ the positive orthant described by elements of $\mathbb{N}_0^{k-1}$ with entries $\geq (N_{1,j},\cdots, N_{k-1,j})$ is not contained in the union of all orthants of previous sequence elements. However, this contradicts Dickson's Lemma, which says that such a sequence of orthants can not be infinite, see e.g. \cite{FigFigSchSch-11}.
\end{proof}

We will now show:
\begin{prop}\label{prop:connectivity-of-descending-links}
    For every $m\in \mathbb{N}$ there is a $r\in \mathbb{N}$ such that the descending link of every vertex $x=[T,f]\in \mathfrak{X}$ with $h(x)\geq r$ is $m$-connected.
\end{prop} 
For this we define
\[
    \alpha:= \max\left\{\alpha(\rho,i)\mid \rho \mbox{ is a collection of } \leq m+2 \mbox{ carets and } 1\leq i \leq k\right\}
\]
(here one should think of $\rho$ as a collection of carets that can form a simplex of dimension $\leq m+1$ in the descending link).  Since for every $A\in \mathbb{N}$ there is a height $r$ such that for all trees at height $\geq r$ we require at least $A$ leaf expansions to obtain them from $T_0$, Proposition \ref{prop:connectivity-of-descending-links} is an immediate consequence of the following result.

\begin{lem}\label{lem:connectivity-of-descending-links}
    For every $m\in \mathbb{N}$ there is a constant $C(m)>0$ such that if $T$ is a tree with history $(N_1,\cdots, N_k)$ consisting of at least $k\cdot (\alpha+C(m))$ leaf expansions, then the descending link of $x=[T,f]$ is $m$-connected for any $f\in {\rm RP}_G(\mathcal{T},\mathcal{B})$. 
\end{lem}
\begin{proof}
    Let $\beta$ be the maximal number of terminal leaves of a caret. Let $C(m)>0$ be a power of $2$ such that $\left\lfloor \frac{C(m)}{2\beta}\right\rfloor-1\geq m$.
    
    We will show that there is a simplex $\sigma$ in ${\rm Lk}^{\downarrow}(x)$ which is $(m+1)$-flag for ${\rm Lk}^{\downarrow}(x)$ and to which we can apply Lemma \ref{lem:zaremsky-connectivity-lemma} to obtain the desired connectivity properties.

    By the pigeon hole principle and our assumptions there is an $i\in \left\{1,\dots, k\right\}$ such that $N_i\geq \alpha + C(m)$, say, $i=1$. 
    
    This means that there is a tree $T'$ with the same history as $T$ which is obtained from $T_0$ by the following sequence of leaf expansions. The construction of $T'$ is illustrated in Figure \ref{fig:tree-T-prime}, where the tree $T_0$ is drawn in black. 
    \begin{enumerate}
        \item Fix one of the $L_i(T_0)\geq 2$ leaves of $T_0$ and perform one leaf expansion of type $1$. This produces a tree $T_1$ with $L_1(T_1)\geq L_1(T_0)+1$ and $L_i(T_1)\geq L_i(T_0)$ for $2\leq i \leq k$ by the viral expansion property. The expansion performed in the construction of $T_1$ from $T_0$ is indicated in red in Figure \ref{fig:tree-T-prime}.
        \item  Fix a type $1$ leaf of $T_1\setminus T_0$ and iteratively perform $C(m)-1$ leaf expansions of type $1$ starting from an expansion in this leaf to produce a tree $T_{\sigma}$ which is an elementary expansion of a subtree in $\frac{C(m)}{2}$ carets of type $1$. The viral expansion property guarantees that this is possible, e.g. by performing the leaf expansions in a binary tree pattern. The expansions that produce $T_{\sigma}$ from $T_1$ are indicated in purple in Figure \ref{fig:tree-T-prime}.
        \item Perform a sequence of leaf expansions that produce a tree $T'$ from $T_{\sigma}$ with history $(N_1,\dots, N_k)$ in such a way that none of the leaves of $T_{\sigma}\setminus T_1$ are expanded. This is possible due to the viral expansion property and our assumption that $N_1\geq \alpha + C(m)\geq C(m)$. The expansions that produce $T'$ from $T_{\sigma}$ are indicated in orange in Figure \ref{fig:tree-T-prime}.
    \end{enumerate}

    \begin{figure}
        \centering
        \includegraphics[width=.7\linewidth]{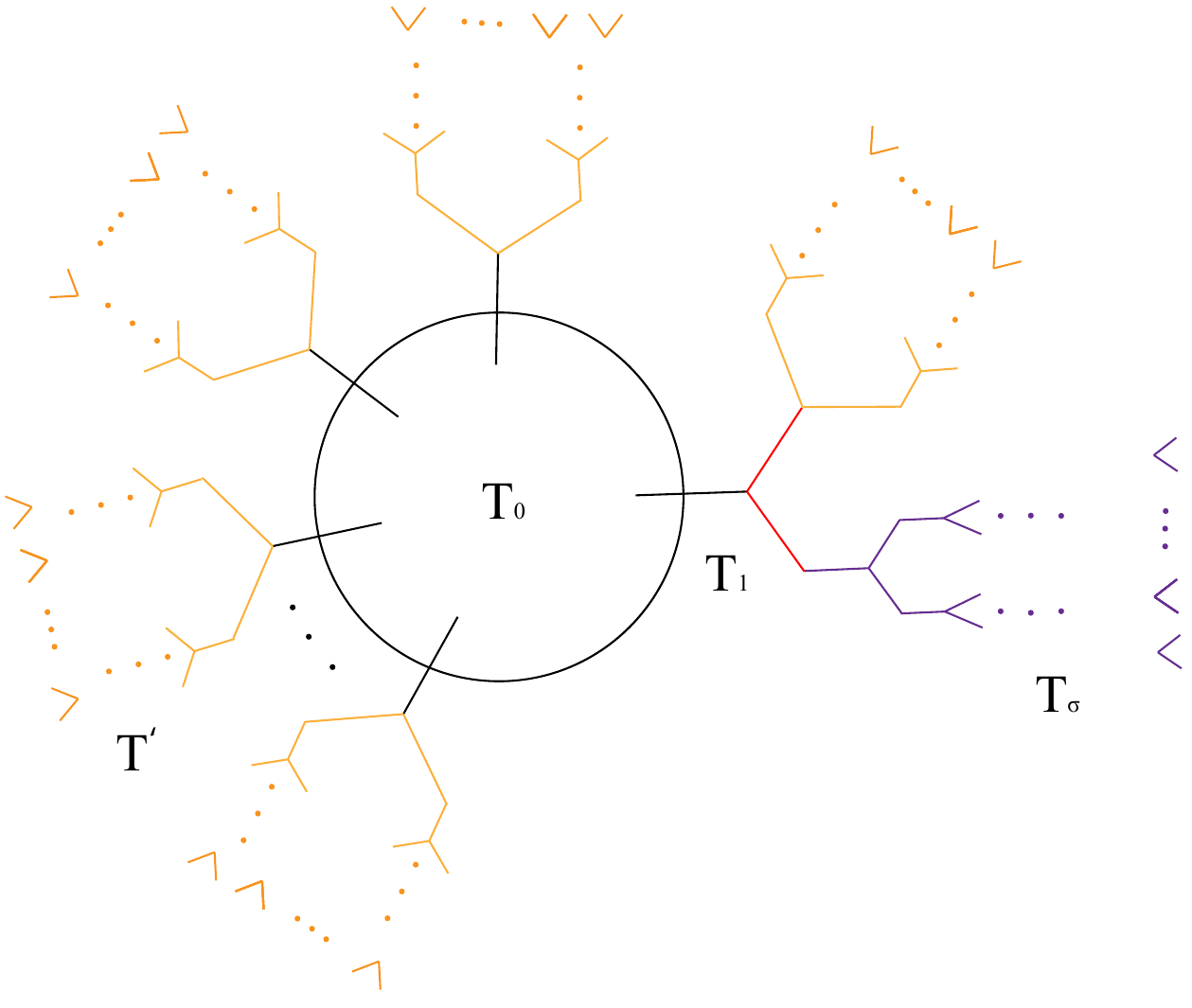}
        \caption{Construction of the tree $T'$ in the proof of Lemma \ref{lem:connectivity-of-descending-links} with Steps (1), (2) and (3) marked in red, purple and orange.}
        \label{fig:tree-T-prime}
    \end{figure}

    By construction there is a collection $\sigma$ of $\frac{C(m)}{2}$ carets of type $1$ such that $T'$ is an elementary expansion of type $\sigma$. Since by Lemma \ref{lem:equivalence-of-pairs} there is some $f'\in {\rm RP}_G(\mathcal{T},\mathcal{B})$ with $[T,f]=[T',f']$, it follows that $\sigma$ defines a $\frac{C(m)}{2}$-simplex in ${\rm Lk}^{\downarrow}(x)$. 
    
    To complete the proof, we will show that we can apply Lemma \ref{lem:zaremsky-connectivity-lemma} to $\sigma$ to deduce that ${\rm Lk}^{\downarrow}(x)$ is $m$-connected. For this it suffices to verify the following two properties:
    \begin{itemize}
        \item[(i)] $\sigma$ is a \emph{$\beta$-ground} for ${\rm Lk}^{\downarrow}(x)$, that is, every vertex of ${\rm Lk}^{\downarrow}(x)$ is connected to all but at most $\beta$ vertices in $\sigma$;
        \item[(ii)] if $\rho$ is a simplex in ${\rm Lk}^{\downarrow}(x)$ of dimension $\ell \leq m+1$ that is vertex-wise joinable to a subsimplex $\sigma'$ of $\sigma$ of dimension $\leq m-\ell$, then $\rho$ is $(m+1)$-joinable to $\sigma'$.
    \end{itemize}
    Indeed, Condition (ii) implies that ${\rm Lk}^{\downarrow}(x)$ is $(m+1)$-flag with respect to $\sigma$, and Condition (i) combined with Condition (ii) implies that every vertex in ${\rm Lk}^{\downarrow}(x)$ is $(m+1)$-joinable to some $\left(\frac{C(m)}{2}-\beta\right)$-pseudoface of $\sigma$. Thus, Lemma \ref{lem:zaremsky-connectivity-lemma} implies that ${\rm Lk}^{\downarrow}(x)$ is $m$-connected, since we chose $C(m)$ such that $\left\lfloor\frac{C(m)}{2\beta}\right\rfloor-1\geq m$.

    \emph{We first prove (ii):} 
     Since $\rho$ is vertex-wise joinable to $\sigma'$, the leaves defining the carets of $\rho$ are disjoint from the leaves defining the carets of $\sigma'$. Since $\rho$ is a simplex of ${\rm Lk}^{\downarrow}(x)$, $T$ can be rearranged to a tree $T_{\rho}$, which is an elementary $\rho$-expansion. By Lemma \ref{lem:maximum-in-alternative-histories} and definition of $\alpha$ we can obtain a tree $S$ which is an elementary $\rho$-expansion \emph{and} a rearrangement of a tree with history $(M_1,\cdots, M_k)$ with $M_1\leq \alpha$. 
     
     We will now use this tree $S$ to construct a tree $R$ which is an elementary $(\rho\sqcup \sigma')$-expansion. We illustrate this construction in Figure \ref{fig:connectivity-of-links}.  
     \begin{itemize}
         \item[(I)] We first grow the binary tree $T_{\sigma}$, using $C(m)$ expansions of type $1$. The expansion producing $T_{\sigma}$ is indicated in red, purple and magenta in Figure \ref{fig:connectivity-of-links}.
         \item[(II)] By (1) the tree $T_{\sigma}$ has $L_i(T_0)$ leaves of type $i$ that are not in $T_{\sigma}\setminus T_1$. Our assumptions on the $N_i$ guarantee that we can use these leaves to perform the sequence of leaf expansions with history $(M_1,\cdots, M_k)$ that are used to produce the tree $S$. This results in a tree $T_2$ that has the elements of $\rho\sqcup \sigma'$ as carets and where none of the leaves of $T_{\sigma}\setminus T_1$ were expanded. In Figure \ref{fig:connectivity-of-links} the extensions performed in this step are drawn in green and orange. 
         \item[(III)] Finally, we perform $N_1-M_1-C(m)\geq 0$ expansions of type $1$ and $N_i-M_i\geq 0$ expansions of type $i$ ``linearly'' in leaves of $T_2$ to obtain a tree $R$ which is a rearrangement of $T'$ (and thus of $T$) and is an elementary expansion in the carets in the disjoint union $\rho\sqcup \sigma'$. In Figure \ref{fig:connectivity-of-links} this is indicated in blue with the new terminal caret $\rho$ still indicated in green.
     \end{itemize}
     By choosing an identification of the leaves of $R$ with the leaves of $x$ that maps the terminal leaves of the carets of $R$ contained in the disjoint union $\rho \sqcup \sigma'$ bijectively to their corresponding leaves in $x$, we deduce that $\rho$ is $(m+1)$-joinable to $\sigma'$ in ${\rm Lk}^{\downarrow}(x)$.
        
    \begin{figure}
        \centering
        \includegraphics[width=.8\linewidth]{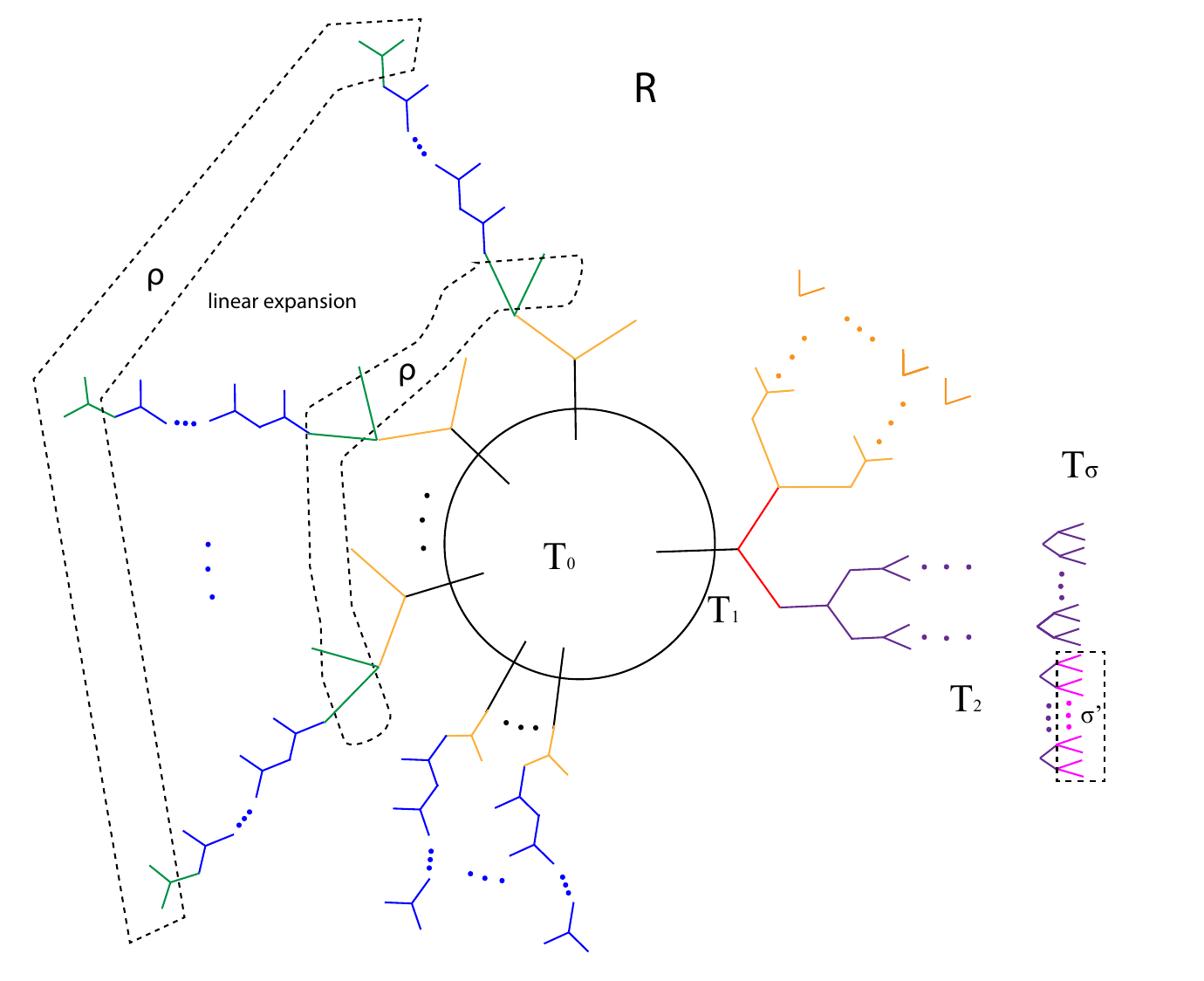}
        \caption{The tree $R$ constructed in the proof of Lemma \ref{lem:connectivity-of-descending-links}. We indicate the reducible carets from $\rho$ (resp. $\sigma'$) in green (resp. magenta).}
        \label{fig:connectivity-of-links}
    \end{figure}

    \emph{The proof of (i) is similar to the proof of (ii):}
    Assume that we have chosen a finite set of leaves of $T$ that form the terminal leaves of a caret $v$ in some rearrangement of $T$ (or equivalently that $v$ is a vertex of ${\rm Lk}^{\downarrow}(x)$). By definition of $\beta$, at most $\beta$ carets corresponding to vertices of $\sigma$ contain one of these leaves. Our definition of $\alpha$ and Lemma \ref{lem:maximum-in-alternative-histories} imply that we can rearrange $T'$ (and thus $T$) as follows:
    \begin{itemize}
        \item We first remove all $\gamma\leq \beta$ carets from $T_{\sigma}$ that contain one of the leaves of $v$, producing a subtree $T_{\sigma,v}\subset T_{\sigma}$ with history $(C(m)-\gamma,0,\dots, 0)$
        \item By definition of $\alpha$ we can then replace the remaining history $(N_1-C(m)+\gamma,N_2,\cdots, N_k)$ by an alternative history that from $T_{\sigma,v}$ produces a tree $S$ which is a rearrangement of $T'$ and has the leaves corresponding to $v$ as terminal leaves of a caret of the desired expansion type. Step (1) of our construction process for $T'$ guarantees that we can obtain $S$ from $T_{\sigma,v}$ without expanding any leaves of the carets that we attached to construct $T_{\sigma,v}$ from $T_1$.
    \end{itemize}
    This proves that $S$ is an elementary expansion of a subtree obtained by attaching the caret $v$, as well as the $\geq \frac{C(m)}{2}-\beta$ carets of $\sigma$ contained in $T_{\sigma,v}$, or equivalently that $v$ is connected in ${\rm Lk}^{\downarrow}(x)$ to all but $\gamma\leq \beta$ vertices in $\sigma$. Since $v$ was an arbitrary vertex of ${\rm Lk}^{\downarrow}(x)$, we proved that $\sigma$ is a $\beta$-ground for ${\rm Lk}^{\downarrow}(x)$.

   This completes the proof of (i) and (ii) and thus the proof of the lemma.  
    \end{proof}

\section{Analysis of the cell stabilisers}
\label{sec:finiteness-properties-of-cell-stabilisers}

In this section we will describe the cell stabilisers of the $\mathrm{RP}_G(\mathcal{T},\mathcal{B})$-action on $\mathfrak{X}$. This will later allow us to determine their finiteness properties. 

\begin{lem}\label{lem:cell-stabilisers}
    Assume that the $G$-action on $\mathcal{T}$ is strongly faithful and let $C=[y,x]$ be a $k$-cube in $\mathfrak{X}$. Then there is a short exact sequence 
    \[
        1\to N \to \mathrm{Stab}_{\mathrm{RP}_G(\mathcal{T},\mathcal{B})}(C)\to Q\to 1
    \]
    where $N$ is isomorphic to a direct product of finitely many edge stabilisers of the $G$-action on $\mathcal{T}$ and $Q$ is a finite group.
\end{lem}
\begin{proof}
    We may assume that $y=[T',f]$, $x=[T,f]$ and $T$ is obtained from $T'$ by attaching $k$ carets to leaves of $T'$. Since $\mathrm{RP}_G(\mathcal{T},\mathcal{B})$ acts transitively on the second entry by post-composition, we may further assume that $y=[T',id_{\mathcal{T}}]$ and $x=[T,id_{\mathcal{T}}]$. 
    
    Then $g\in \mathrm{Stab}_{\mathrm{RP}_G(\mathcal{T},\mathcal{B})}(C)$ if and only if $g(T)=T$, $g$ is supported on $T$, and $g$ permutes the carets that we need to attach to leaves of $T$ to obtain $T'$ preserving caret types and without breaking any of their edges. In particular, $\mathrm{Stab}_{\mathrm{RP}_G(\mathcal{T},\mathcal{B})}(C)$ acts on $V(T)$ by permutations that map interior vertices to interior vertices and preserve the $\mathfrak{G}$-colouring of the leaves of $T$. Let $N\unlhd \mathrm{Stab}_{\mathrm{RP}_G(\mathcal{T},\mathcal{B})}(C)$ be the kernel of this action, and let $Q$ be the finite quotient group describing the action on $T$. Assume now that $g\in N$. Then $g|_{T}=id_T$ and thus $g$ fixes the leaves of $T$ pointwise. Since $G$ acts on $\mathcal{T}$ without inversions and the action is strongly faithful, $g$ can be identified with an element of the direct product of the $G$-stabilisers of the edges corresponding to the leaves of $T$. Conversely, every such element induces a unique element of $N$ by extending by the identity on $T$. This completes the proof.
\end{proof}

\begin{cor}\label{cor:finiteness-properties-of-cell-stabilisers}
    If the $G$-action on $\mathcal{T}$ is strongly faithful and all its edge stabilisers are of type $F_n$, then all cell stabilisers of the $\mathrm{RP}_G(\mathcal{T},\mathcal{B})$-action on $\mathfrak{X}$ are of type $F_n$.
\end{cor}
\begin{proof}
    This is an immediate consequence of Lemma \ref{lem:cell-stabilisers} and the fact that a group that is virtually a direct product of finitely many groups of type $F_n$ is of type $F_n$.
\end{proof}

\section{Oligomorphic actions}
\label{sec:oligomorphic-actions}
To deduce Theorem \ref{thm:main-actions-on-trees} from the above results, we apply results of Belk--Zaremsky on twisted Brin--Thompson groups. For every group of permutations $H$ of a set $S$, in \cite{BelZar-20} Belk and Zaremsky introduce a group $SV_H$, which they call the twisted Brin--Thompson group. They prove that $SV_H$ is an infinite simple group that contains $H$ as a subgroup and use it to prove new embedding results of groups in infinite simple groups. For instance, their result plays a key role in the recent proof of the Boone--Higman Conjecture for hyperbolic groups by Belk, Bleak, Matucci and Zaremsky \cite{BBMZ-hyperbolic-23}. Here we will use their work to prove that the Boone--Higman Conjecture also holds for interesting classes of graphs of groups, that we call generalised Baumslag--Solitar groups. To state the result from \cite{BelZar-20} that we need, we require the following definition.
\begin{defn}\label{defn:oligmc}
    An action of a group $H$ on a set $S$ is called \emph{oligomorphic} if for every $k\in \mathbb{N}$ there are only finitely many $H$-orbits of $k$-tuples. 
\end{defn}

\begin{thm}[{Belk--Zaremsky}]\label{thm:Belk-Zaremsky}
    Let $H$ be a group that acts by oligomorphic permutations on a set $S$. Let $n\in \mathbb{N}\cup \left\{\infty\right\}$. If $H$ is of type $F_n$, and the stabiliser in $H$ of any finite subset of $S$ is of type $F_n$, then $SV_H$ is of type $F_n$. In particular, $H$ is a subgroup of a simple group of type $F_n$.
\end{thm}
\begin{proof}
    This is an immediate consequence of \cite{BelZar-20}. The first part is \cite[Theorem D]{BelZar-20}. The ``in particular''-part follows from the fact that by \cite[Theorem 3.4]{BelZar-20} $SV_H$ is simple and contains $H$ as a subgroup, see \cite[Section 3]{BelZar-20}.
\end{proof}

\begin{lem}\label{lem:RP-is-oligomorphic}
    Let $\Gamma$ be a graph equipped with a faithful $G$-action. Then $G$ is a subgroup of $\mathrm{RP}_{G}(\Gamma)$ and the $\mathrm{RP}_{G}(\Gamma)$-action on $V(\Gamma)$ is oligomorphic.
\end{lem}
\begin{proof}
    If $G$ acts faithfully on $\Gamma$, then it follows from the definition of $\mathrm{RP}_{G}(\Gamma)$ that $G$ is a subgroup. Since by definition $\mathrm{RP}_{G}(\Gamma)$ contains all finitely supported permutations, for every $k\geq 1$ the $\mathrm{RP}_{G}(\Gamma)$-action on $V(\Gamma)$ has finitely many $\mathrm{RP}_{G}(\Gamma)$-orbits of $k$-tuples, implying that the action is oligomorphic.
\end{proof}

\begin{cor}\label{cor:rp-embeds-in-simple-group}
    Let $\Gamma$ be a graph equipped with a faithful $G$-action and let $n\in \mathbb{N}\cup \left\{\infty\right\}$. If $\mathrm{RP}_G(\Gamma,\mathcal{B})$ is $F_n$ for all finite subsets $\mathcal{B}$ of $V(\Gamma)$, then $\mathrm{RP}_{G}(\Gamma)$ embeds in a simple group of type $F_n$.
\end{cor}
\begin{proof}
    This is an immediate consequence of Theorem \ref{thm:Belk-Zaremsky} and Lemma \ref{lem:RP-is-oligomorphic}.
\end{proof}

\section{Main Theorem about actions on trees}\label{sec:mainthm-action-on-tree}
We now have all the ingredients required to state and prove our main result. 
\begin{thm}\label{thm:main-strongly-faithful-actions-on-trees}
    Let $G$ be a group that acts strongly faithfully and without inversions on an infinite tree $\mathcal{T}$ without leaves such that all edge stabilisers are of type $F_n$ and such that there is a system of gates $\mathfrak{G}$ and a tree $T_0$ for which $(\mathrm{RP}_G(\mathcal{T}),\mathfrak{G},T_0)$ has the viral expansion property. Then $\mathrm{RP}_G(\mathcal{T})$ embeds in a simple group of type $F_n$.
\end{thm}

The main step in the proof of Theorem \ref{thm:main-strongly-faithful-actions-on-trees} is the following intermediate result that is of interest in its own right.
\begin{prop}\label{prop:main-strongly-faithful-actions-on-trees}
    Let $G$ be a group that acts strongly faithfully and without inversions on a tree $\mathcal{T}$. Assume that all edge stabilisers are of type $F_n$ and that there is a system of gates $\mathfrak{G}$ and an admissible tree $T_0\subset \mathcal{T}$ such that $(\mathrm{RP}_G(\mathcal{T}),\mathfrak{G},T_0)$ has the viral expansion property. Then for every finite subset $\mathcal{B}\subset \mathcal{T}$ the rigid permutation group $\mathrm{RP}_G(\mathcal{T},\mathcal{B})$ is of type $F_n$.
\end{prop}
\begin{proof}
    We start by observing that for every finite subset $\mathcal{B}\subset \mathcal{T}$ we can choose an admissible base tree $T_{0,\mathcal{B}}$ containing $\mathcal{B}$ and $T_0$. This ensures that the triple $(\mathrm{RP}_G(\mathcal{T},\mathcal{B}),\mathfrak{G},T_{0,\mathcal{B}})$ also has the viral expansion property. By construction, for all $\mathcal{B}$ the $\mathrm{RP}_G(\mathcal{T},\mathcal{B})$-action on the sublevel sets of the Stein--Farley complex associated with $\mathcal{B}$ and $T_{0,\mathcal{B}}$ is cocompact. Thus, the result follows by combining Theorem \ref{thm:Brown-and-BB} with Proposition \ref{prop:connectivity-of-descending-links} and Corollary \ref{cor:finiteness-properties-of-cell-stabilisers}.
\end{proof}

\begin{proof}[Proof of Theorem \ref{thm:main-strongly-faithful-actions-on-trees}]
 The result is an immediate consequence of Proposition \ref{prop:main-strongly-faithful-actions-on-trees} and Corollary \ref{cor:rp-embeds-in-simple-group}.
\end{proof}

\section{Application to graphs of groups}
\label{sec:graphs-of-groups}

The main application of our techniques is that graphs of groups whose edge groups have finite index in the vertex groups and whose action on the associated Bass--Serre tree is faithful satisfy the Boone--Higman Conjecture. In this section we will prove this result. To do so, we first summarize some results from Bass--Serre theory that we will need. We refer to \cite{Ser-03,ScottWall-79} for two accounts on this subject.

A \emph{(finite) graph of groups} $\mathcal{G}_{\Gamma}$ is a finite connected graph $\Gamma$ where all vertices $v$ and edges $e$ are labelled by groups $G_v$ and $G_e$, and for every edge $e$ there are injective homomorphism $\phi_{e,\iota} : G_e\hookrightarrow G_{\iota(e)},~ \phi_{e,\tau} : G_e\hookrightarrow G_{\tau(e)}$, where $\iota(e)$ is the initial vertex  and $\tau(e)$ is the terminal vertex of $e$. A graph of groups defines a group, called its fundamental group $G=\pi_1(\mathcal{G}_{\Gamma})$. It is the group that arises as fundamental group of the space where one chooses for every $G_v$ and $G_e$ a classifying space and glues them along continuous maps that induce the inclusions of the edge groups in the vertex groups. 

This description allows us to choose a presentation for $G$ as follows. Let $T\subseteq \Gamma$ be a maximal tree and let $G_v\cong \left\langle X_v\mid R_v\right\rangle$. Then $G\cong \langle X \mid R\rangle$, where 
\[
X=\left(\bigsqcup_{v\in V(\Gamma)} X_v\right)\sqcup \left\{e\mid e\in E(\Gamma\setminus T)\right\}
\]
and
\[
    R=\left(\bigsqcup_{v\in V(\Gamma)} R_v\right) \sqcup \left\{ e\cdot \phi_{e,\iota}(g)\cdot e^{-1} = \phi_{e,\tau}(g)\mid e\in E(\Gamma\setminus T),~g\in G_e\right\}.
\]
This presentation depends on the choice of a maximal subtree $T$ and we therefore also denote it by $\pi_1(\mathcal{G}_{\Gamma},T)$.

In the following we will sometimes just refer to $G$ as a graph of groups and drop the words ``fundamental group of''.

The main result of Bass--Serre theory says that decompositions of $G$ as finite graphs of groups are in correspondence with cocompact $G$-actions on trees that do not invert any edges. More precisely, if $G$ acts on a tree $\mathcal{T}$ cocompactly and without edge inversions, then we can give it the structure of a graph of groups with underlying graph $\Gamma=\mathcal{T}/G$. We denote $\pi:\mathcal{T}\to \Gamma$ the projection. The vertex (resp. edge) label of $v\in V(\Gamma)$ (resp. $e\in E(\Gamma)$) is isomorphic to $\mathrm{Stab}_{G}(\widetilde{v})$ (resp. $\mathrm{Stab}_{G}(\widetilde{e})$) for $\widetilde{v}$ (resp. $\widetilde{e}$) any lift of $v$ (resp. $e$) with respect to $\pi$, and the morphisms $\phi_{e,\iota}$ and $\phi_{e,\tau}$ are induced by inclusions of the edge stabilisers in the vertex stabilisers in $\mathcal{T}$.

Conversely, the Bass--Serre tree $\mathcal{T}$ associated with the presentation $\pi_1(\mathcal{G}_\Gamma,T)$ for a maximal tree $T\subset \mathcal{G}$ is defined by the following vertex set, edge set, and attaching maps:
\[
    V(\mathcal{T})=\bigsqcup_{v\in V(\Gamma)} G/G_v, ~ E(\mathcal{T})=\bigsqcup_{e\in E(\Gamma)} G/G_e,
\]
\[
    \iota(g\cdot G_e)= g\cdot G_{\iota(e)},~ \tau(g\cdot G_e)=\left\{\begin{array}{ll} g\cdot e\cdot G_{\tau(e)},& \mbox{ if } e\in E(\Gamma\setminus T)\\ g\cdot G_{\tau(e)},& \mbox{ else. }\end{array}\right.
\]
The $G$-action on $\mathcal{T}$ is by left-multiplication. We observe that the degree of a vertex $\widetilde{v}$ in $\mathcal{T}$ is equal to $\sum_{e\in E(\Gamma),~\iota(e)=v}[G_v:G_e] + \sum_{e\in E(\Gamma),~\tau(e)=v}[G_v:G_e]$. In particular, if $\Gamma$ is a finite graph, then $\mathcal{T}$ is locally finite if and only if the indices of all edge group inclusions in the vertex groups are finite. One can show that this construction does not depend on the choices made up to $G$-equivariant isomorphism. In particular, up to isomorphism the Bass--Serre tree of a graph of groups does not depend on a choice of a maximal tree $T$, justifying why we call it ``the'' Bass--Serre tree of $\mathcal{G}_{\Gamma}$.

Let $\mathcal{G}_\Gamma$ be a graph of groups. A \emph{subgraph of groups} $\mathcal{H}_\Lambda\leq \mathcal{G}_\Gamma$ is a connected subgraph $\Lambda\leq \Gamma$ whose vertices are labelled by subgroups $H_v\leq G_v$ of the vertex groups of $\Gamma$ and whose edges are labelled by the intersection $H_e=G_e\cap \phi_{e,\iota}^{-1}(H_{\iota(e)})\cap \phi_{e,\tau}^{-1}(H_{\tau(e)})$. We can choose spanning trees $T_{\Lambda}$ for $\Lambda$ and $T_{\Gamma}$ for $\Gamma$ such that $T_{\Lambda}\subseteq T_{\Gamma}$. Subsequently, when discussing results about the Bass--Serre tree of a subgraph of groups, we will implicitly assume that it has been constructed with respect to such a choice of spanning trees. It is a straight-forward consequence of the normal form theorems for elements of fundamental groups of graphs of groups (see \cite[Theorem~11 and Corollary~1, page~45]{Ser-03}) that the embeddings of edge and vertex groups induce a canonical embedding $\pi_1(\mathcal{H}_\Lambda)\hookrightarrow \pi_1(\mathcal{G}_\Gamma)$. The following result shows that if we want to check the faithfulness of the action of a graph of groups on its Bass--Serre tree, it suffices to find a suitable subgraph of groups that acts faithfully. This will be useful later, when we apply our methods to explicit examples.

\begin{lem}\label{lem:subgraphs-of-groups}
    With the above notation assume that $\mathcal{H}_\Lambda \leq \mathcal{G}_\Gamma$ is a subgraph of groups so that for every edge $e\in E(\Lambda)$ we have $H_e=G_e$. Denote $H=\pi_1(\mathcal{H}_\Lambda)$, $G=\pi_1(\mathcal{G}_\Gamma)$ and denote by $\mathcal{T}_{\Lambda}$ and $\mathcal{T}_{\Gamma}$ the corresponding Bass--Serre trees. If $H$ acts faithfully on $\mathcal{T}_{\Lambda}$, then $G$ acts faithfully on $\mathcal{T}_{\Gamma}$.
\end{lem}
\begin{proof}
     There is a canonical $H$-equivariant embedding $f \colon \mathcal{T}_{\Lambda}\hookrightarrow \mathcal{T}_{\Gamma}$ of the corresponding Bass--Serre trees defined by $h\cdot H_v\mapsto h\cdot G_v$ and $h\cdot H_e\mapsto h\cdot G_e$ for all $v\in V(\Lambda)$, $h\in E(\Lambda)$. By construction, its image is the subtree of $\mathcal{T}_{\Gamma}$ spanned by the $H$-orbits of $\bigsqcup_{v\in V(\Lambda)} G_v$ and $\bigsqcup_{e\in E(\Lambda)} G_e$.

    Since $G$ acts without inversions, the kernel of the $G$-action on $\mathcal{T}_{\Gamma}$ is equal to the intersection $\bigcap _{e\in E(\mathcal{T}_{\Gamma})} {\rm Stab}_G(e)$. The latter is a subgroup of $$\bigcap_{e\in E(f(\mathcal{T}_{\Lambda}))} {\rm Stab}_G(e) = \bigcap_{e\in E(f(\mathcal{T}_{\Lambda}))} {\rm Stab}_H(e),$$ where the last equality follows from the assumption that $H_e=G_e$ for all $e\in E(\Lambda)$. The right hand side of this equality is isomorphic to the kernel of the $H$-action on $\mathcal{T}_{\Lambda}$, which is trivial by assumption. Thus, $G$ acts faithfully on $\mathcal{T}_{\Gamma}$. This completes the proof. 
\end{proof}

We now assume that $G$ is a graph of groups that acts faithfully and cocompactly on its Bass--Serre tree $\mathcal{T}$ with quotient $\Gamma$. We further assume that $\mathcal{T}$ is locally finite and that all edge stabilisers are of type $F_n$ for $n\geq 2$. Our goal is to show that then $G$ embeds in a simple group of type $F_n$ and thus satisfies the Boone--Higman Conjecture. To achieve this, we aim to apply our results on rigid permutation groups from the previous sections. This requires us to find an admissible system of gates $\mathfrak{G}$ on $\Gamma$ together with a suitable base tree $T_0$ such that $(\mathrm{RP}_{G}(\mathcal{T}),\mathfrak{G},T_0)$ satisfies the viral expansion property. However, it is easy to see that if we start with a general graph of groups, then it can happen that no such system exists. To see this consider $\mathbb{Z}$ viewed as an HNN-extension of the trivial group and observe that its Bass--Serre tree is a bi-infinite line. The issue is that the indices of the edge groups in the vertex groups are too small. To resolve this we observe that $G$ can be embedded in a bigger graph of groups $\widehat{G}$ defined by the following data: 
\[
    \widehat{G}_v:=G_v\times \mathbb{Z}/3\mathbb{Z},~\widehat{G}_e:=G_e,~\widehat{G}_e=G_e\hookrightarrow G_v\times \left\{1\right\}\leq \widehat{G}_v \mbox{ for } v=\iota(e),\tau(e),
\]
where the inclusions are induced by the inclusions for $G$.

We will call $\widehat{G}$ the \emph{augmented graph of groups} associated with $G$ and denote by $\widehat{\mathcal{T}}$ the Bass--Serre tree of $\widehat{G}$. Observe that the vertex and edge groups of $\widehat{G}$ have the same finiteness properties as those of $G$, since they are commensurable. As discussed above, the embeddings $G_e\hookrightarrow \widehat{G}_e$ and $G_v\hookrightarrow \widehat{G}_v$ induce a canonical embedding $G\hookrightarrow \widehat{G}$ together with a $G$-equivariant embedding $\mathcal{T}\hookrightarrow \widehat{\mathcal{T}}$. 

 \begin{lem}\label{lem:strongly-faithful-action}
        If $G$ acts faithfully on $\mathcal{T}$, then $\widehat{G}$ acts strongly faithfully on $\widehat{\mathcal{T}}$.
    \end{lem}

    \begin{proof}
        We identify $\mathcal{T}$ with its canonical embedding in $\widehat{\mathcal{T}}$ defined above.  We first show that every piece $R\subset \widehat{\mathcal{T}}$ contains a translate of $\mathcal{T}$. To see this, choose a vertex $w$ in $R$ and let $v$ be a vertex of the underlying finite graph such that $w$ lies in the same $\widehat{G}$-orbit as the vertex $1\cdot \widehat{G}_v$. Then there is an element $g\in \widehat{G}$ such that $g \widehat{G}_v = w$. Now if $g\mathcal{T}$ already lies inside $R$, we are done. Otherwise, let $e_1$ be the unique edge that cuts off $R$ from $\widehat{\mathcal{T}}$ and let  $e_2$ be the unique edge with one vertex $w$ and pointing towards $e_1$. Then $e_2$ has to be contained in $g\mathcal{T}$. Let $h\in g\widehat{G}_vg^{-1}= {\rm Stab}_{\widehat{G}}(w)$ be an element that is not contained in the subgroup $gG_vg^{-1}$. Then $hg\mathcal{T}$ will still contain the vertex $w$ but not the edge $e_2$, thus $hg\mathcal{T}$ must completely lie in $R$. 
        
        Now the intersection of the edge stabilisers of the $\widehat{G}$-action on $hg\mathcal{T}$ is isomorphic (via conjugation by $hg$) to the intersection of all edge stabilisers of the $G$-action on $\mathcal{T}$, which is trivial by the faithfulness of the $G$-action on $\mathcal{T}$. Thus, the restriction of the $\widehat{G}$-action on $\widehat{\mathcal{T}}$ to every piece $R$ is faithful. Hence, the action is strongly faithful.
    \end{proof}

    We define a system of gates $\mathfrak{G}$ on $\Gamma$ by choosing at least one half-edge of every $e\in E(\Gamma)$ as follows: We fix an identification $V(\Gamma)=\left\{1,\cdots, n\right\}$ and then define that for every edge $e\in E(\Gamma)$ the set $\mathfrak{G}$ contains the half-edge corresponding to the vertex of $e$ with larger integer value if the end points of $e$ are distinct. If $e$ defines a loop in $\Gamma$ we add both of its half-edges to $\mathfrak{G}$. Figure \ref{fig:system-of-gates-for-a-finite-graph} provides an example of a system of gates obtained from this procedure after identifying the vertices $v_i$ with $i$.

    \begin{lem}
        $\mathfrak{G}$ defines an admissible system of gates on $\widehat{\mathcal{T}}$.
    \end{lem}
    \begin{proof}
        Our choice of gates guarantees that every edge path of length at least $\max\{n,2\}$ in $\Gamma$ will exit some edge through a half-edge in $\mathfrak{G}$. Indeed, this is clear if the path backtracks at some point. Moreover, our assumption on the length guarantees that if it does not backtrack, then a subpath will define a loop in $\Gamma$ and our choice of gates guarantees that every simple closed loop in $\Gamma$ exits at least one edge through a half-edge in $\mathfrak{G}$. Thus, every simple path of length at least $\max\left\{n,2\right\}$  in $\widehat{T}$ will exit some edge through a half-edge labelled by an element of $\mathfrak{G}$. Hence, K\"onig's Lemma \cite{Koe-27} implies that $\mathfrak{G}$ defines an admissible system of gates on $\widehat{\mathcal{T}}$.
    \end{proof}

    \begin{lem}\label{lem:viral-exp-for-augmented-graph-of-groups}
        For every finite subset $\mathcal{B}\subset \widehat{\mathcal{T}}$ there is an admissible subtree $T_0$ such that the triple $(\mathrm{RP}_{\widehat{G}}(\widehat{\mathcal{T}},\mathcal{B}), \mathfrak{G}, T_0)$ has the viral expansion property.
    \end{lem}
    \begin{proof}
                \begin{figure}[h]
        \centering
        \includegraphics[width=.6\linewidth]{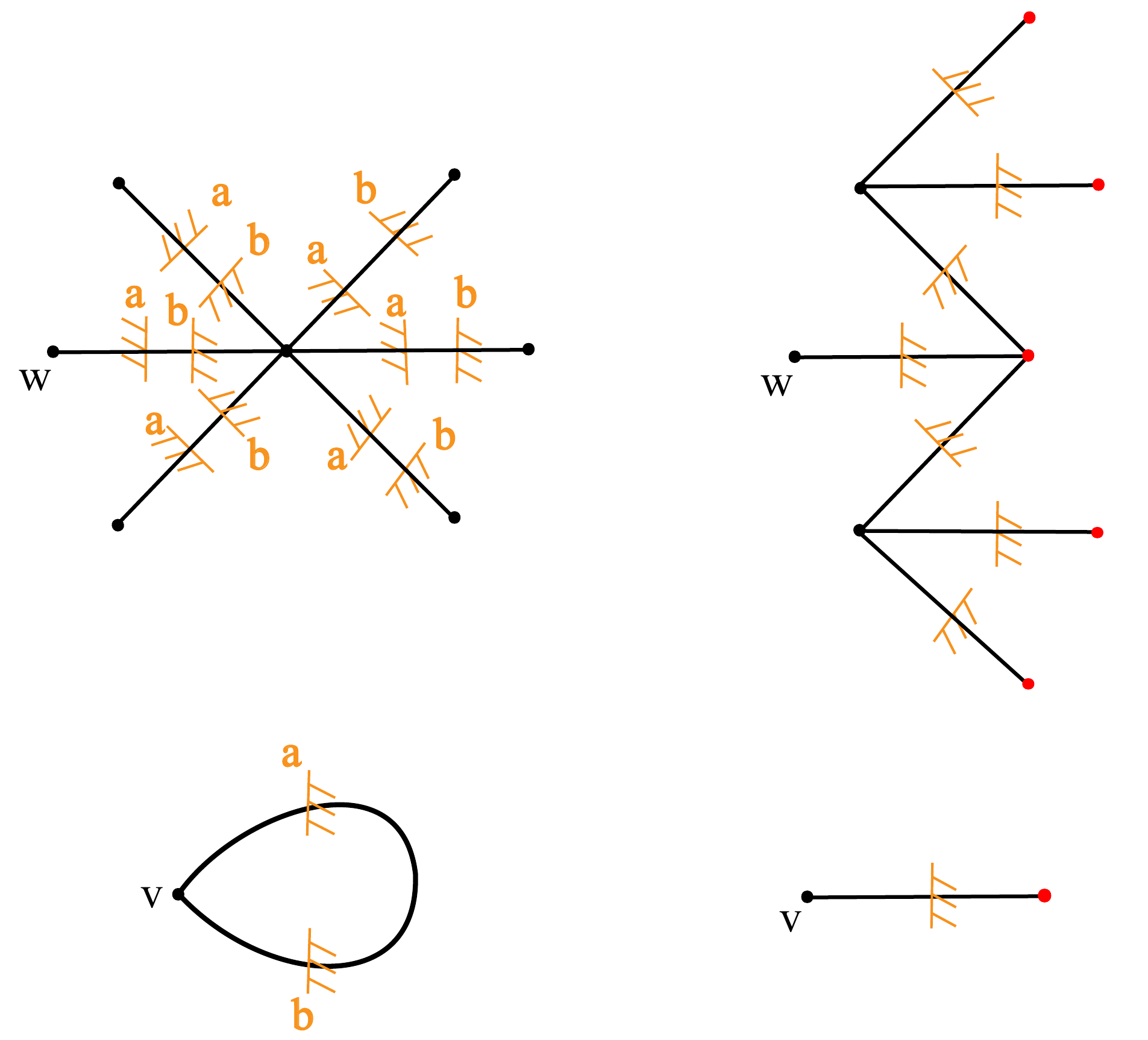}
        \caption{Examples of carets in an augmented graph of groups}
        \label{fig:example-carets}
    \end{figure}
        By hypothesis, we have $[\widehat{G}_{\iota(e)}\colon \widehat{G}_e] \geq 3$ and $[\widehat{G}_{\tau(e)}\colon \widehat{G}_e]\geq 3$ for each edge $e\in \Gamma$. Thus, for every $v\in V(\Gamma)$, 
        every preimage $w$ of $v$ in $\widehat{\mathcal{T}}$ and every edge $e\in E(\Gamma)$ with $\iota(e)=v$ (resp. $\tau(e)=v$), the half-edge of $e$ corresponding to the inclusion $G_{e}\hookrightarrow G_{\iota(e)}$ (resp. $G_{e}\hookrightarrow G_{\tau(e)}$) has at least $3$ preimages in $\widehat{\mathcal{T}}$ adjacent to $w$. This implies that if $S_i$ is a caret corresponding to an edge which does (resp. does not) define a loop in $\Gamma$, then $S_i$ has at least three (resp. four) terminal leaves of type $i$.
        Indeed, they are obtained by lifting to $\widehat{\mathcal{T}}$ the path in $\Gamma$ that traverses the loop twice starting through the half edge not labelled by $i$ (resp. backtracks along the edge with half edge labelled by $i$). See Figure \ref{fig:example-carets} for a picture illustrating the two cases.  Thus, $M_{ii}\geq 3$ for all $i$ and by Remark \ref{rem:viral-expansion-property} a sufficiently large choice of admissible base tree $T_0$ provides us with a triple $(\mathrm{RP}_{\widehat{G}}(\widehat{\mathcal{T}},\mathcal{B}), \mathfrak{G}, T_0)$ which has the viral expansion property; note that our choice of gates was minimal, meaning that we do not have to pass to a subset of $\mathfrak{G}$ when applying Remark \ref{rem:viral-expansion-property}.
    \end{proof}

    As a consequence we obtain the following result, which is the main application of our techniques.

    \begin{thm}\label{thm:main-graphs-of-groups}
        Let $G$ be a graph of groups such that all vertex groups are of type $F_n$ and all edge groups have finite index in the vertex groups. Assume that $G$ acts faithfully on its Bass--Serre tree. Then $G$ embeds in a simple group of type $F_n$. 
    \end{thm}
    \begin{proof}
        This is an immediate consequence of Theorem \ref{thm:main-strongly-faithful-actions-on-trees}, Lemma \ref{lem:strongly-faithful-action}, Lemma \ref{lem:viral-exp-for-augmented-graph-of-groups} and the fact that the edge stabilisers of the $\widehat{G}$-action on $\widehat{\mathcal{T}}$ are isomorphic to the edge stabilisers of the $G$-action on $\mathcal{T}$, thus $F_n$.
    \end{proof}

    \section{Generalised Baumslag--Solitar groups}
    In this section for every finitely presented group $G$ we will introduce a very general class of graphs of groups $\BSol_G$ to which our methods can be applied, which we will call \emph{generalised Baumslag--Solitar groups over $G$}. We will then prove that if for a given $G$ there is at least one group in $\BSol_G$ that acts faithfully on its associated Bass--Serre tree, then all groups in $\BSol_G$ satisfy the Boone--Higman Conjecture. 
    
    The notion of a \emph{generalised Baumslag--Solitar group} was coined by Forester for the class of (finite) graphs of groups with infinite cyclic edge groups \cite{For-03}, although the geometry of this class of groups was already studied earlier, see \cite{Whyte01}. This notion has since been generalised by Button \cite{But-22} to the classes of (finite) graphs of groups with vertex and edge groups isomorphic to $\BZ^n$ such that all edge inclusions have finite index; he calls such a group a \emph{generalised Baumslag--Solitar groups of rank $n\geq 1$}.

    Here we introduce a further generalisation. Given a group $G$ we denote by $\BSol_G$ the class of groups consisting of all (finite) graphs of groups with the property that all edge and vertex groups are abstractly commensurable with $G$ and all edge group inclusions in vertex groups have finite index. We call the groups in $\BSol_G$ \emph{generalised Baumslag--Solitar groups over $G$}. For every element of $\BSol_G$ we fix a defining decomposition as graph of groups. We can now state and prove the main result of this section, before giving several applications in Section \ref{sec:applications}.
    \begin{thm}\label{thm:main-BS}
        Let $G$ be a group of type $F_n$ for $n\geq 2$. Assume that there is a non-trivial group $H\in \BSol_{G}$ such that $H$ acts faithfully on its Bass--Serre tree. Then every $K\in \BSol_{G}$ embeds in a simple group of type $F_n$. In particular, $K$ satisfies the Boone--Higman Conjecture. 
    \end{thm}
    \begin{proof}
        Denote by $\mathcal{H}_{\Gamma}$ and $\mathcal{K}_{\Lambda}$ the graphs of groups for $H$ and $K$. Let $v\in V(\Gamma)$ and $w\in V(\Lambda)$ be arbitrary vertices. Since commensurability is an equivalence relation on groups, there are finite index subgroups $H_{0}\leq H_v$ and $K_{0}\leq K_w$ with $H_0\cong K_0$. Let $\Delta$ be the graph obtained by gluing $\Gamma$ and $\Lambda$ via a new edge $e_{vw}$ that connects $v$ to $w$, that is, $V(\Delta)=V(\Gamma)\sqcup V(\Lambda)$, $E(\Delta)=E(\Gamma)\sqcup E(\Lambda) \sqcup \left\{e_{vw}\right\}$, and $\iota(e_{vw})=v,~ \tau(e_{vw})=w$. We equip $\Delta$ with a graph of groups structure $\mathcal{L}_{\Delta}$ by equipping the subgraph $\Gamma$ (resp. $\Lambda)$ with the graph of groups structure $\mathcal{H}_{\Gamma}$ (resp. $\mathcal{K}_{\Lambda}$) and labelling the edge $e_{vw}$ by the group $H_0\cong K_0$ together with the canonical inclusions $H_0\hookrightarrow H_v$ and $K_0\hookrightarrow K_w$. By construction $L=\pi_1(\mathcal{L}_{\Delta})\in \BSol_G$. Moreover, $\mathcal{H}_{\Gamma}\leq \mathcal{L}_{\Delta}$ is a subgraph of groups whose fundamental group acts faithfully on its associated Bass--Serre tree. Thus, it follows from Lemma \ref{lem:subgraphs-of-groups} that $L$ acts faithfully on its Bass--Serre tree. Since all vertex and edge groups are commensurable to $G$ and thus of type $F_n$, Theorem \ref{thm:main-graphs-of-groups} implies that $L$ embeds in a simple group of type $F_n$. Since $n\geq 2$, and $K$ is a subgroup of $L$, this completes the proof.
    \end{proof}

    \section{Applications to explicit examples}\label{sec:applications}
    Theorem \ref{thm:main-BS} reduces the Boone--Higman Conjecture for generalised Baumslag--Solitar groups over a given finitely presented group $G$ to finding a single non-trivial group $H\in \BSol_G$ which acts faithfully on its associated Bass--Serre tree. Here we show that such groups can be found in many interesting cases, including when $G$ is one of the following groups: $\BZ^n$, a nonabelian free group $\mathbb{F}_n$, and many nilpotent groups (such as the Heisenberg group). This allows us to show the Boone--Higman Conjecture for many interesting examples for which it was still open, including Baumslag--Solitar groups, Leary--Minasyan groups and Free-by-cyclic groups (see \cite[Problem 5.3]{BBMZ-survey-23}).  

    \subsection{Torsion-free strongly scale-invariant groups} 
    A group $G$ is called \emph{strongly scale-invariant} if there is an injective endomorphism $f: G\to G$ such that $\bigcap _{i>0} f^i(G)$ is finite. As a consequence of Theorem \ref{thm:main-BS} we can prove the following result.

    \begin{thm}\label{thm:scale-invariant}
        Let $G$ be a torsion-free strongly scale-invariant group of type $F_n$ for $n\geq 2$. Then every group in $\BSol_G$ embeds in a simple group of type $F_n$. In particular, every group in $\BSol_G$ satisfies the Boone--Higman Conjecture.
    \end{thm}
    \begin{proof}
        Let $f: G\to G$ be an injective homomorphism such that $\bigcap _{i>0} f^i(G)$ is finite and thus trivial, since $G$ is torsion-free. Let $H= \left\langle G, t\mid tgt^{-1}=f(g), ~ g \in G\right\rangle$ be the HNN-extension of $G$ defined by $f$. We claim that the action of $H$ on the associated Bass--Serre tree is faithful. Indeed, any element of $h\in H$ that acts trivially, in particular has to stabilise all edges. Thus, $h$ has to be contained in the intersection 
        \[
            \bigcap_{i>0} t^i G t^{-i} = \bigcap_{i>0} f^i(G)=\left\{1\right\}
        \] 
        of edge stabilisers, which is trivial. The assertion then follows from Theorem \ref{thm:main-BS}. 
    \end{proof}

    A conjecture attributed to Nekrashevych and Pete \cite{NekPet-11, Der-22} asserts that every finitely generated strongly scale-invariant group is virtually nilpotent. Conversely, it is known that many (but not all) finitely generated nilpotent groups are strongly scale-invariant. In the next sections we will discuss some concrete examples that allow us to deduce the Boone--Higman Conjecture for several classes of groups of interest in geometric group theory; we refer to \cite{Der-22} for further details, examples and references regarding strongly scale-invariant groups.

    \subsection{Baumslag--Solitar groups}

    For all integers $n,m\in \mathbb{Z}$, the Baumslag--Solitar group $BS(n,m)$ is the group defined by the finite presentation
    \[
        BS(n,m):= \left\langle a, b \mid ba^nb^{-1}=a ^m \right\rangle.
    \]
    
    If $n,m\neq 0$, it is the HNN-extension with vertex and edge group $\mathbb{Z}$, where we identify the subgroup $n\mathbb{Z}\leq \mathbb{Z}$ with the subgroup $m\mathbb{Z}\leq \mathbb{Z}$. In particular, $BS(n,m)\in \BSol_{\mathbb{Z}}$ if $n,m\neq 0$.

    \begin{thm}\label{thm:gen-BS}
        Every group in $\BSol_{\mathbb{Z}}$ embeds in a simple group of type $F_{\infty}$. In particular, the Boone--Higman Conjecture holds for $BS(n,m)$ for all $n,m\in \mathbb{Z}$.
    \end{thm}
    \begin{proof}
        If $n=0$ or $m=0$, $BS(n,m)$ is virtually free and hence embeds in Thompson's group V which is type $F_\infty$ \cite{Bro-87} and simple. If $n,m\neq 0$, $BS(n,m)\in \BSol_{\mathbb{Z}}$ and, by Theorem \ref{thm:scale-invariant}, it suffices to show that $\mathbb{Z}$ is strongly scale-invariant. This follows by observing that the injective endomorphism $f: \mathbb{Z}\to \mathbb{Z}$, $n\mapsto 2n$ satisfies
        \[
            \bigcap_{i>0} f^i(\mathbb{Z}) = \bigcap_{i>0} 2^i \mathbb{Z} = \left\{0\right\}.
        \]
    \end{proof}

    Theorem \ref{thm:gen-BS} proves that in particular the Boone--Higman Conjecture holds for all generalised Baumslag--Solitar groups, that is, finite graphs of groups all of whose edge and vertex groups are infinite cyclic. We also emphasize that Baumslag--Solitar groups are one of the classes of groups raised explicitly in \cite{BBMZ-survey-23} as groups for which it would be interesting to prove the Boone--Higman Conjecture.

    \subsection{Leary--Minasyan groups}

    In \cite{But-22} Button introduces the notion of a \emph{generalised Baumslag--Solitar group of rank $k$} as a finite graph of groups all of whose vertex and edge groups are isomorphic to $\mathbb{Z}^k$. A key motivation for his generalisation is that this class includes the Leary--Minasyan groups \cite{LeaMin-21}, which included the first examples of ${\rm CAT}(0)$, but not virtually biautomatic groups. The generalised Baumslag--Solitar groups of rank $k$ are contained in $\BSol_{\mathbb{Z}^k}$. An analogous proof as the one of Theorem \ref{thm:gen-BS} shows:
    \begin{thm}\label{thm:gen-BS-all-ranks}
        For all $k\geq 1$ every group in $\BSol_{\mathbb{Z}^k}$ embeds in a simple group of type $F_{\infty}$.
    \end{thm}
    \begin{proof}
        By Theorem \ref{thm:scale-invariant} it suffices to prove that $\mathbb{Z}^k$ is strongly scale-invariant. This follows as in the proof of Theorem \ref{thm:gen-BS} by considering the injective endomorphism that maps $\mathbb{Z}^k$ to the subgroup $(2\mathbb{Z})^k\leq \mathbb{Z}^k$.
    \end{proof}
    
    As a consequence we obtain the following result.

    \begin{cor}\label{cor:Leary-Minasyan-groups}
        All Leary--Minasyan groups constructed in \cite{LeaMin-21} satisfy the Boone--Higman Conjecture. In particular, there is a ${\rm CAT}(0)$ group that is not virtually biautomatic and satisfies the Boone--Higman Conjecture.
    \end{cor}
    \begin{proof}
        By construction the Leary--Minasyan groups are HNN-extensions of $\mathbb{Z}^n$ that identify two finite index subgroups. Thus, the assertion follows from Theorem \ref{thm:gen-BS-all-ranks}.
    \end{proof}

    \subsection{Examples of higher nilpotency class}

    As we mentioned, it is well-known that many finitely generated nilpotent groups are strongly scale-invariant. By Theorem \ref{thm:scale-invariant} this provides large classes of torsion-free finitely generated groups $G$ of type $F_{\infty}$ for which all groups in $\BSol_G$ satisfy the Boone--Higman Conjecture. 
    
    To illustrate this we give a few more concrete examples to which we can thus apply Theorem \ref{thm:scale-invariant}. 
    
    Our first example is the integral Heisenberg group 
    $H_{2k+1}(\mathbb{Z})$ of dimension $2k+1$,
    for $k\geq 1$. For instance, for 
    $H_3(\mathbb{Z})=\left\langle x, y, z \mid [x,y]=z, [x,z]=[y,z]=1\right\rangle$ we can define an injective morphism $f$ that shows that it is strongly scale-invariant by $f(x)= x^2$, $f(y)=y^2$, $f(z)=z^4$.
    
    More generally, for the model filiform $(k-1)$-nilpotent group
    \[
        \Lambda_k= \left\langle x_1, x_2, \dots, x_k \mid  [x_1,x_i]=x_{i+1},~ 2\leq i \leq k-1, \mbox{ and } [x_j,x_l]=1 \mbox{ else}
        \right\rangle 
    \]
    of dimension $k$, we can choose $f(x_1)=x_1^2$ and $f(x_i)=x_i^{2^{i-1}}$, for $2\leq i \leq k$, to show that it is strongly scale-invariant.
    
    \subsection{Free-by-cyclic groups}
    From the preceding examples it may seem that allowing for edge and vertex groups that are commensurable to $G$ in the definition of $\BSol_G$ might solely serve the purpose of choosing the largest possible class of graphs of groups to which the conclusion of Theorem \ref{thm:main-BS} applies. However, we will now illustrate with the case when $G=\mathbb{F}_k$, $k\geq 2$, is $k$-generated non-abelian free that sometimes considering this generalised class is even necessary to be able to verify the assumptions of Theorem \ref{thm:main-BS}. 

    \begin{thm}\label{thm:free-groups}
        Let $G=\mathbb{F}_k$ for $k\geq 2$ be $k$-generated non-abelian free. Then every group in $\BSol_{\mathbb{F}_k}$ embeds in a simple group of type $F_{\infty}$. In particular, non-abelian finitely generated free by cyclic groups satisfy the Boone--Higman Conjecture.
    \end{thm}

    To prove Theorem \ref{thm:free-groups} we need a group in $\BSol_{\mathbb{F}_k}=\BSol_{\mathbb{F}_2}$ that acts faithfully on its Bass--Serre tree. Examples of such groups are provided by the Burger--Mozes groups \cite{BurMoz-97, BurMoz-00}.

    \begin{thm}[{Burger--Mozes \cite[Theorem 5.5]{BurMoz-00}}]\label{thm:Burger-Mozes}
    There is an amalgamated product of the form $K=\mathbb{F}_n\ast_{H} \mathbb{F}_n$ with the following properties:
    \begin{enumerate}
        \item $K$ is simple, in particular $K$ acts faithfully on the corresponding Bass-Serre tree;
       
        \item $K$ is of type $F_{\infty}$;
       
        \item $[\mathbb{F}_n\colon H]<\infty$.

        \item \label{thm:amal-free-faithful-tree:conj} Denote the two inclusion maps $H\to \mathbb{F}_n$ by $\pi_1$ and $\pi_2$, then there is an element $\phi\in Aut(H)$ such that $\pi_2=\pi_1\circ \phi$.
    \end{enumerate}  
\end{thm}

    \begin{proof}[Proof of Theorem \ref{thm:free-groups}]
        Since all non-abelian finitely generated free groups are commensurable, we have $\BSol_{\mathbb{F}_k}=\BSol_{\mathbb{F}_2}$. By Theorem \ref{thm:main-BS} it thus suffices to give a single example of a group in $\BSol_{\mathbb{F}_2}$ that acts faithfully on its Bass--Serre tree. Such an example is provided by the Burger--Mozes groups, see Theorem \ref{thm:Burger-Mozes}. Thus, all groups in $\BSol_{\mathbb{F}_2}$ embed in a simple group of type $F_{\infty}$. Since non-abelian finitely generated free by cyclic groups are HNN extensions defined by an automorphism of a non-abelian finitely generated free group, they are contained in $\BSol_{\mathbb{F}_2}$. This completes the proof. 
    \end{proof}

    \begin{rem}\label{rem:burger-mozes}
        Our proof of Theorem \ref{thm:free-groups} shows that for every group in $\BSol_{\mathbb{F}_k}$  the permutational Boone--Higman Conjecture  holds, that is, every such group embeds in a finitely presented group that admits an action of type (A) on a countable set in the sense of Zaremsky \cite[Conjecture 5.1]{Zar-24}. In particular, this provides a positive answer to \cite[Question 5.13(ii)]{Zar-24} for the Burger--Mozes groups. 
    \end{rem}
   Squier proved in \cite{Squ87} that all the Euclidean triangle Artin groups also lie in $\BSol_{\mathbb{F}_2}$, hence we have the following:
   
\begin{cor}\label{cor-eu-triart}
    All Euclidean triangle Artin groups embed in finitely presented simple groups of type $F_\infty$. 
\end{cor}

   Combined with the other results of this section, this proves Theorem \ref{mainthm:Boone-Higman}.

\bibliographystyle{alpha}
\bibliography{references.bib}

\newcommand{\etalchar}[1]{$^{#1}$}
\begin{thebibliography}{BBMZ23b}

\bibitem[AAF18]{AuAyFa18}
Samuel Audino, Delaney~R. Aydel, and Daniel Farley.
\newblock Quasiautomorphism groups of type {$F_\infty$}.
\newblock {\em Algebr. Geom. Topol.}, 18(4):2339--2369, 2018.

\bibitem[ABF{\etalchar{+}}21]{AraBuxFlePetWu-21}
Javier Aramayona, Kai-Uwe Bux, Jonas Flechsig, Nansen Petrosyan, and Xiaolei
  Wu.
\newblock Asymptotic mapping class groups of {C}antor manifolds and their
  finiteness properties.
\newblock {\em arXiv preprint arXiv:2110.05318}, 2021.

\bibitem[BB97]{BesBra-97}
M.~Bestvina and N.~Brady.
\newblock Morse theory and finiteness properties of groups.
\newblock {\em Invent. Math.}, 129(3):445--470, 1997.

\bibitem[BBMZ23a]{BBMZ-hyperbolic-23}
James Belk, Collin Bleak, Francesco Matucci, and Matthew~CB Zaremsky.
\newblock Hyperbolic groups satisfy the {Boone-Higman} conjecture.
\newblock {\em arXiv preprint arXiv:2309.06224}, 2023.

\bibitem[BBMZ23b]{BBMZ-survey-23}
James Belk, Collin Bleak, Francesco Matucci, and Matthew~CB Zaremsky.
\newblock Progress around the {B}oone-{H}igman conjecture.
\newblock {\em arXiv preprint arXiv:2306.16356}, 2023.

\bibitem[BH74]{BoHi74}
William~W. Boone and Graham Higman.
\newblock An algebraic characterization of groups with soluble word problem.
\newblock {\em J. Austral. Math. Soc.}, 18:41--53, 1974.

\bibitem[BM97]{BurMoz-97}
Marc Burger and Shahar Mozes.
\newblock Finitely presented simple groups and products of trees.
\newblock {\em C. R. Acad. Sci. Paris S\'{e}r. I Math.}, 324(7):747--752, 1997.

\bibitem[BM00]{BurMoz-00}
Marc Burger and Shahar Mozes.
\newblock Lattices in product of trees.
\newblock {\em Inst. Hautes \'{E}tudes Sci. Publ. Math.}, (92):151--194, 2000.

\bibitem[Bro87]{Bro-87}
Kenneth~S. Brown.
\newblock Finiteness properties of groups.
\newblock In {\em Proceedings of the {N}orthwestern conference on cohomology of
  groups ({E}vanston, {I}ll., 1985)}, volume~44, pages 45--75, 1987.

\bibitem[But22]{But-22}
Jack~O Button.
\newblock Generalised {B}aumslag-{S}olitar groups and hierarchically hyperbolic
  groups.
\newblock {\em arXiv preprint arXiv:2208.12688}, 2022.

\bibitem[BZ22]{BelZar-20}
James Belk and Matthew C.~B. Zaremsky.
\newblock Twisted {B}rin-{T}hompson groups.
\newblock {\em Geom. Topol.}, 26(3):1189--1223, 2022.

\bibitem[Der22]{Der-22}
Jonas Der\'{e}.
\newblock Strongly scale-invariant virtually polycyclic groups.
\newblock {\em Groups Geom. Dyn.}, 16(3):985--1004, 2022.

\bibitem[FFSS11]{FigFigSchSch-11}
Diego Figueira, Santiago Figueira, Sylvain Schmitz, and Philippe Schnoebelen.
\newblock Ackermannian and primitive-recursive bounds with {D}ickson's lemma.
\newblock In {\em 26th {A}nnual {IEEE} {S}ymposium on {L}ogic in {C}omputer
  {S}cience---{LICS} 2011}, pages 269--278. IEEE Computer Soc., Los Alamitos,
  CA, 2011.

\bibitem[For03]{For-03}
Max Forester.
\newblock On uniqueness of {JSJ} decompositions of finitely generated groups.
\newblock {\em Comment. Math. Helv.}, 78(4):740--751, 2003.

\bibitem[GLU22]{GenLonUre-22}
Anthony Genevois, Anne Lonjou, and Christian Urech.
\newblock Asymptotically rigid mapping class groups, {I}: {F}initeness
  properties of braided {T}hompson's and {H}oughton's groups.
\newblock {\em Geom. Topol.}, 26(3):1385--1434, 2022.

\bibitem[GPS99]{GTS99}
Thierry Giordano, Ian~F. Putnam, and Christian~F. Skau.
\newblock Full groups of {C}antor minimal systems.
\newblock {\em Israel J. Math.}, 111:285--320, 1999.

\bibitem[Hig61]{Hig-61}
G.~Higman.
\newblock Subgroups of finitely presented groups.
\newblock {\em Proc. Roy. Soc. London Ser. A}, 262:455--475, 1961.

\bibitem[K{\"o}n27]{Koe-27}
D{\'e}nes K{\"o}nig.
\newblock {\"U}ber eine schlussweise aus dem endlichen ins unendliche.
\newblock {\em Acta Sci. Math.(Szeged)}, 3(2-3):121--130, 1927.

\bibitem[Kuz58]{Kuznetsov58}
Alexander Kuznetsov.
\newblock Algorithms as operations in algebraic systems.
\newblock In {\em Proceedings of the All-Union Colloquium on Algebra Article 2
  from the morning session of 6th February 1958}, volume 13.3, pages 240--241,
  1958.

\bibitem[Leh08]{Lehnert2008}
J{\"o}rg Lehnert.
\newblock {\em Gruppen von quasi-Automorphismen}.
\newblock doctoralthesis, Universit{\"a}tsbibliothek Johann Christian
  Senckenberg, 2008.

\bibitem[LM21]{LeaMin-21}
Ian~J. Leary and Ashot Minasyan.
\newblock Commensurating {HNN} extensions: nonpositive curvature and
  biautomaticity.
\newblock {\em Geom. Topol.}, 25(4):1819--1860, 2021.

\bibitem[Mat12]{Mat12}
Hiroki Matui.
\newblock Homology and topological full groups of \'{e}tale groupoids on
  totally disconnected spaces.
\newblock {\em Proc. Lond. Math. Soc. (3)}, 104(1):27--56, 2012.

\bibitem[NP11]{NekPet-11}
Volodymyr Nekrashevych and G\'{a}bor Pete.
\newblock Scale-invariant groups.
\newblock {\em Groups Geom. Dyn.}, 5(1):139--167, 2011.

\bibitem[NSJG18]{NucStJG-18}
Brita E.~A. Nucinkis and Simon St. John-Green.
\newblock Quasi-automorphisms of the infinite rooted 2-edge-coloured binary
  tree.
\newblock {\em Groups Geom. Dyn.}, 12(2):529--570, 2018.

\bibitem[Qui78]{Qui-78}
Daniel Quillen.
\newblock Homotopy properties of the poset of nontrivial p-subgroups of a
  group.
\newblock {\em Advances in Mathematics}, 28(2):101--128, 1978.

\bibitem[Ser03]{Ser-03}
Jean-Pierre Serre.
\newblock {\em Trees}.
\newblock Springer Monographs in Mathematics. Springer-Verlag, Berlin, 2003.
\newblock Translated from the French original by John Stillwell, Corrected 2nd
  printing of the 1980 English translation.

\bibitem[Squ87]{Squ87}
Craig~C. Squier.
\newblock On certain {$3$}-generator {A}rtin groups.
\newblock {\em Trans. Amer. Math. Soc.}, 302(1):117--124, 1987.

\bibitem[SW79]{ScottWall-79}
Peter Scott and Terry Wall.
\newblock Topological methods in group theory.
\newblock In {\em Homological group theory ({P}roc. {S}ympos., {D}urham,
  1977)}, volume~36 of {\em London Math. Soc. Lecture Note Ser.}, pages
  137--203. Cambridge Univ. Press, Cambridge-New York, 1979.

\bibitem[Why01]{Whyte01}
K.~Whyte.
\newblock The large scale geometry of the higher {B}aumslag-{S}olitar groups.
\newblock {\em Geom. Funct. Anal.}, 11(6):1327--1343, 2001.

\bibitem[Zar17]{Zar17}
Matthew C.~B. Zaremsky.
\newblock On the {$\Sigma$}-invariants of generalized {T}hompson groups and
  {H}oughton groups.
\newblock {\em Int. Math. Res. Not. IMRN}, (19):5861--5896, 2017.

\bibitem[Zar24a]{Zar-24}
Matthew C.~B. Zaremsky.
\newblock Finite presentability of twisted brin-thompson groups.
\newblock {\em arXiv preprint arXiv:2405.18354}, 2024.

\bibitem[Zar24b]{Zar-24a}
Matthew~C.B. Zaremsky.
\newblock Embedding finitely presented self-similar groups into finitely
  presented simple groups.
\newblock {\em arXiv preprint arXiv:2405.09722}, 2024.

\end{thebibliography}

\end{document}